\documentclass[11pt]{article}

\usepackage{amsfonts,amssymb,amsmath,amsthm,epsfig,euscript}

\newtheorem{theorem}{Theorem}
\newtheorem{lemma}[theorem]{Lemma}
\newtheorem{corollary}[theorem]{Corollary}
\newtheorem{definition}[theorem]{Definition}

\long\def\symbolfootnote[#1]#2{\begingroup
\def\thefootnote{\fnsymbol{footnote}}\footnote[#1]{#2}\endgroup}

\newcommand{\la}{\lambda}
\newcommand{\La}{\Lambda}
\newcommand{\ris}[1]{\mathrm{ris}(#1)}

\newcommand{\des}[1]{\mathrm{des}(#1)}

\newcommand{\sg}{\sigma}

\newcommand{\cref}[1]{Corollary \ref{corollary:#1}}

\newcommand{\red}[1]{\mathrm{red}(#1)}

\newcommand{\Plap}[1]{\text{$P$-$\mathrm{nlap}$}(#1)}
\newcommand{\Pmch}[1]{\text{$P$-$\mathrm{mch}$}(#1)}

\newcommand{\Umch}[1]{\text{$\Upsilon$-$\mathrm{mch}$}(#1)}
\newcommand{\Ulap}[1]{\text{$\Upsilon$-$\mathrm{nlap}$}(#1)}

\newcommand{\fig}[2]{\begin{figure}[ht]
\centerline{\scalebox{1.25}{\epsfig{file=#1.eps}}}
\caption{#2}
\label{figure:#1}
\end{figure}}

\title{Patterns in column strict fillings of rectangular arrays.}

\author{
Johannes Harmse \\
\small Department of Mathematics\\[-0.8ex]
\small University of California, San Diego\\[-0.8ex]
\small La Jolla, CA 92093-0112. USA\\[-0.8ex]
\small \texttt{jharmse@math.ucsd.edu}
\and
Jeffrey Remmel\footnote{Partially supported by NSF grant DMS 0654060.} \\
\small Department of Mathematics\\[-0.8ex]
\small University of California, San Diego\\[-0.8ex]
\small La Jolla, CA 92093-0112. USA\\[-0.8ex]
\small \texttt{remmel@math.ucsd.edu}
\and
}

\date{\small Submitted: Date 1;  Accepted: Date 2;
 Published: Date 3.\\
\small MR Subject Classifications: 05A15, 05E05}

\begin{document}
\maketitle

\begin{abstract}
\noindent In this paper, we study pattern matching in the 
set $\mathcal{F}_{n,k}$ of 
fillings of the $k \times n$ rectangle with the integers  
$1, \ldots, kn$ such that the elements in any column 
increase from bottom to top.  Let $P$ be a column strict tableau 
of shape $2^k$. We say that a filling $F \in \mathcal{F}_{n,k}$ 
has $P$-match starting at $i$ if 
the elements of $F$ in columns $i$ and $i+1$ have the same relative order 
as the elements of $P$. We compute the  generating functions 
for the distribution of $P$-matches and nonoverlapping $P$-matches for 
various classes of standard tableaux of shape $2^k$. 
We say that 
a filling $F \in \mathcal{F}_{n,k}$ 
is {\it $P$-alternating} if there are 
$P$-matches  of $F$ starting at all odd positions but there 
are no $P$-matches of $F$ starting at even positions. We also 
compute the  generating functions for $P$-alternating elements 
of $\mathcal{F}_{n,k}$ for various classes of standard tableaux of 
shape $2^k$. 
\end{abstract}

\section{Introduction}

Let $\mathcal{F}_{n,k}$ denote the set of all fillings 
of a $k \times n$ rectangular array with the integers $1, \ldots, 
kn$ such that that the elements increase from bottom to top 
in each column. We let $(i,j)$ denote the cell in 
$i$-th row from the bottom and $j$-th column from the left of the $k \times n$ rectangle 
and we let $F(i,j)$ denote the element in cell $(i,j)$ of 
$F\in\mathcal{F}_{n,k}$. For example, the elements 
of $\mathcal{F}_{2,3}$ are pictured below.

\fig{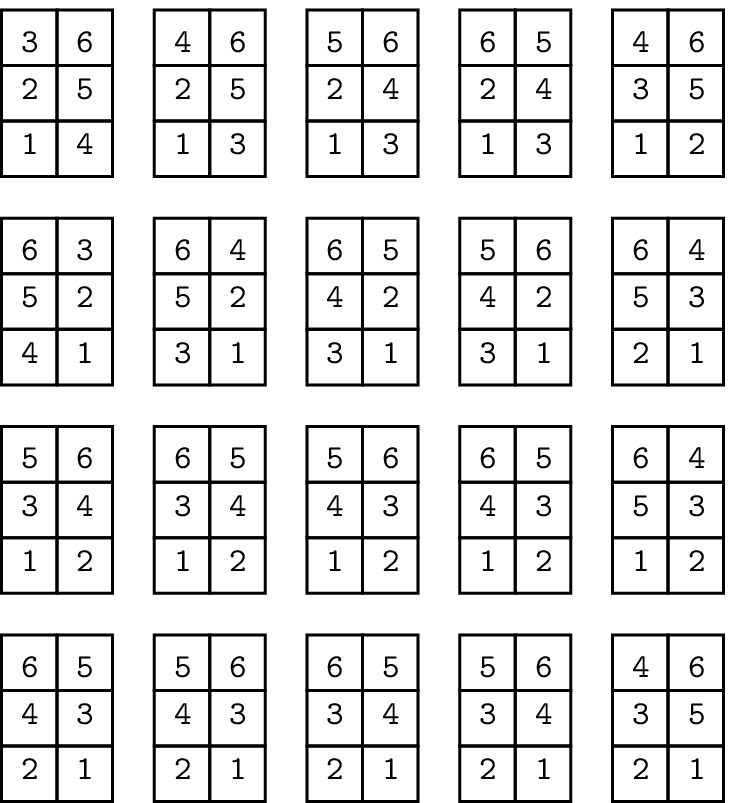}{The elements of $\mathcal{F}_{2,3}$.}

It is easy to see that 
\begin{equation}\label{number}
|\mathcal{F}_{n,k}| = \frac{(kn)!}{(k!)^n}.
\end{equation}
That is, for each $F \in \mathcal{F}_{n,k}$, allowing all permutations of the elements in each column gives rise to
$(k!)^n$ fillings of the $k \times n$ rectangle with the 
numbers $1, \ldots, kn$. Since there are $(kn)!$ fillings of the $k \times n$ rectangle with the 
numbers $1, \ldots, kn$, (\ref{number}) easily follows.

Given a partition 
$\lambda = (\lambda_1, \ldots , \lambda_k)$ where 
$0 < \lambda_1 \leq \cdots \leq \lambda_k$, we let 
$F_\lambda$ denote the Ferrers diagram of $\lambda$, i.e. $F_\lambda$ is 
the set of left-justified rows of squares where the size 
of the $i$-th row, reading from top to bottom, is $\lambda_i$. Thus 
a $k \times n$ rectangular array corresponds to the Ferrers 
diagram corresponding to $n^k$. 
If $F \in \mathcal{F}_{n,k}$ and the integers are  increasing 
in each row, reading from left to right, then $F$ is a 
standard tableau of shape $n^k$. 
We let $St_{n^k}$ denote the set of all standard tableaux of shape 
$n^k$ and let $st_{n^k} =|St_{n^k}|$.  One can use the Frame-Robinson-Thrall 
hook formula \cite{FRT} to show that 
\begin{equation}
st_{n^k} = \frac{(kn)!}{\prod_{i=0}^{k-1} (i+n)\downarrow_n}
\end{equation}
where $(n)\downarrow_0 =1$ and $(n)\downarrow_k =n(n-1) \cdots (n-k+1)$ 
for $k > 0$.

The goal of this paper is to study pattern matching conditions in 
$\mathcal{F}_{n,k}$. Clearly, when $k=1$, we have $\mathcal{F}_{n,1} =S_n$, 
where $S_n$ is the symmetric group, so our 
results can be viewed as generalizations of results on pattern 
matching in $S_n$. Our long term goal is to extend 
the work of this paper to the 
study pattern matching conditions in standard tableaux of rectangular 
shape as well as pattern matching conditions in column strict tableaux  
of rectangular shapes over a fixed alphabet $A =\{1, \ldots, m\}$. 
The study of pattern matching conditions in $\mathcal{F}_{n,k}$ is 
considerably easier than the study of pattern matching conditions in standard 
tableaux or column strict tableax. Nevertheless, we shall see 
that the study of pattern matching conditions in $\mathcal{F}_{n,k}$
requires us to prove some interesting results on pattern matching 
conditions in standard tableaux. In particular, we 
shall generalize two classical results on permutations. Given 
a permutation $\sg = \sg_1 \ldots \sg_n \in S_n$, we let 
\begin{eqnarray*}
Rise(\sg) &=& \{i:\sg_i < \sg_{i+1}\} \ \mbox{and} \\
Des(\sg) &=&  \{i:\sg_i > \sg_{i+1}\}.
\end{eqnarray*}
We let $\ris{\sg} = |Rise(\sg)|$ and $\des{\sg} = |Des(\sg)|$. 
Then the generating function for the distribution of rises or descents 
in $S_n$ is given by 
\begin{equation}\label{risgf}
\sum_{n \geq 0} \frac{t^n}{n!} \sum_{\sg \in S_n} x^{\ris{\sg}} = 
\frac{1-x}{-x+e^{t(x-1)}},
\end{equation}
see Stanley \cite{Stan}.
A permutation $\sg =\sg_1 \ldots \sg_n \in S_n$ is alternating 
if 
$$\sg_1 < \sg_2 > \sg_3 < \sg_4 > \sg_5 < \cdots$$
 or, equivalently, 
if $Ris(\sg)$ equals the set of odd numbers which are less than 
$n$. Let $A_n$ denote the number of alternating permutations 
of $S_n$. Then Andr\'{e} \cite{Andre1}, \cite{Andre2} proved that 
\begin{equation}\label{alte}
\sec(t) = 1+ \sum_{n \geq 1} \frac{A_{2n}t^{2n}}{(2n)!}
\end{equation}
and 
\begin{equation}\label{alto}
\tan(t) = \sum_{n\geq 0} \frac{A_{2n+1}t^{2n+1}}{(2n+1)!}.
\end{equation}

If $F$ is any filling of a $k \times n$-rectangle with distinct positive 
integers such that elements in each column increase, 
reading from bottom to top, then 
we let $\red{F}$ denote the element of 
$\mathcal{F}_{n,k}$ which results from $F$ by replacing 
the $i$-th smallest element of $F$ by $i$. For example, Figure \ref{figure:red} demonstrates 
a filling, $F$, with its corresponding reduced filling, $\red{F}$.

\fig{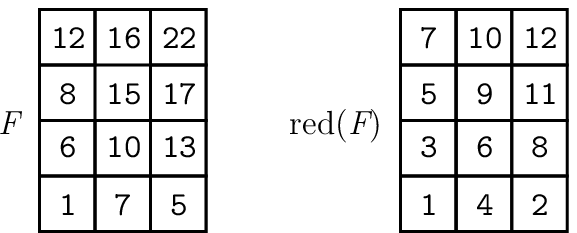}{An example of $F\in\mathcal{F}_{3,4}$ and $\red{F}$.}

If $F \in \mathcal{F}_{n,k}$ and $1 \leq c_1 < \cdots < c_j \leq n$, then 
we let $F[c_1,\ldots,c_j]$ be the filling of the $k \times j$ rectangle 
where the elements in column $a$ of $F[c_1,\ldots,c_j]$ equal the elements 
in column $c_a$ in $F$ for $a = 1, \ldots, j$.   We can 
then extend the usual pattern matching definitions from permutations 
to elements of  $\mathcal{F}_{n,k}$ as follows. 
\begin{definition}\label{def1} Let $P$ be an element of 
$\mathcal{F}_{j,k}$ and $F \in \mathcal{F}_{n,k}$ where $j \leq n$. 
Then we say
\begin{enumerate}
\item  $P$ {\bf occurs} in $F$ if there are
$1 \leq i_1 < i_2 < \cdots < i_j \leq n$ such that
$\red{F[i_1, \ldots, i_j]} = P$,

\item $F$ {\bf avoids} 
$P$ if there is no occurrence of $P$ in $F$, and   

\item there is a {\bf $P$-match  in $F$ starting at
position $i$} if $\red{F[i,i+1, \ldots, i+j-1]} = P$.
\end{enumerate}
\end{definition}
Again, when $k=1$, then $\mathcal{F}_{n,1}=S_n$, where $S_n$ is the symmetric 
group, and our definitions reduce to the standard definitions that 
have appeared in the pattern matching literature. We note 
that Kitaev, Mansour, and Vella \cite{K1} have studied pattern matching 
in matrices which is a more general setting than the one we are considering 
in this paper.

We let $\Pmch{F}$ denote the 
number of $P$-matches in $F$ and let 
$\Plap{F}$ be the maximum number of nonoverlapping $P$-matches in
$F$, where two $P$-matches are said to overlap if they share a common 
column. For example, if we consider the fillings $P \in \mathcal{F}_{3,3}$ and 
$F,G \in \mathcal{F}_{6,3}$ shown in Figure \ref{figure:Pmatch}, 
then it is easy to see that there are no $P$-matches in $F$ 
but there is an occurrence of $P$ in $F$, since 
$\red{F[1,2,5]} =P$. Also, there are 2 $P$-matches in $G$ starting at 
positions 1 and 2, respectively, so $\Pmch{G} =2$ and 
$\Plap{G} =1$.

\fig{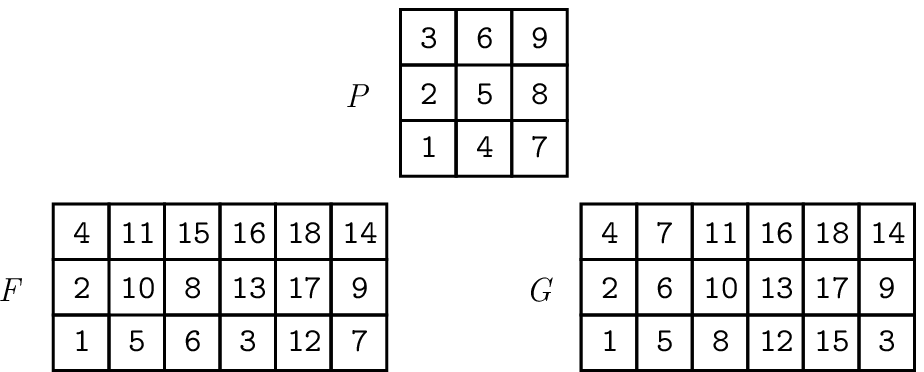}{Examples of $P$-matches and occurrences of $P$.}

One can easily extend these notions to sets of elements 
of $\mathcal{F}_{j,k}$.  
That is, suppose that $\Upsilon \subseteq \mathcal{F}_{j,k}$. 
Then $F \in \mathcal{F}_{n,k}$ has an $\Upsilon$-match at place $i$
provided $\red{F[i,i+1, \ldots,i+j-1]} \in \Upsilon$. Let $\Umch{F}$ and $\Ulap{F}$ be 
the number of $\Upsilon$-matches and nonoverlapping
$\Upsilon$-matches in $F$, respectively.

To generalize (\ref{risgf}), (\ref{alte}), and (\ref{alto}), one 
must first generalize the notion of a rise in a permutation 
$\sg\in S_n$. As one 
may view a rise in $\sg$ as a pattern match of the permutation $12$, a natural  
analogue of a rise in $\mathcal{F}_{n,k}$ is a   
pattern match of one or more patterns, $\Upsilon\subseteq\mathcal{F}_{2,k}$, which have strictly increasing 
rows, i.e. a match of a standard tableaux of shape $k \times 2$.
Consequently, there are many analogues of rises in our setting. 
For example, when $k =2$, the two standard tableaux of shape $2^2$, 
which we denote $P_1^{(2,2)}$ and $P_2^{(2,2)}$, are shown in 
Figure \ref{figure:2standard}. 
Thus, when considering the analogues of a
rise in $\mathcal{F}_{n,2}$, we should 
study matches of the patterns $P_1^{(2,2)}$ and $P_2^{(2,2)}$ as well as matches of 
the set of patterns $St_{2^2} = \{P_1^{(2,2)},P_2^{(2,2)}\}$.   
 
\fig{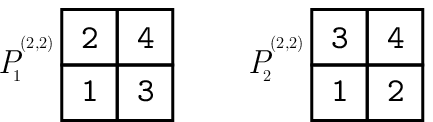}{The standard tableaux of shape $2^2$.}

If $P \in \mathcal{F}_{2,k}$, then 
we define $Full_{n}^{P}$ to be the set of 
$F \in \mathcal{F}_{n,k}$ with $\Pmch{F} = n-1$, i.e. the set of 
$F \in \mathcal{F}_{n,k}$ with the property that 
there are
$P$-matches in $F$ starting at positions $1,2, \ldots, n-1$.  
We let $full_n^P = |Full_n^P|$, and by convention, define $full_{1}^{P}=1$.
For example, if $P$ is the element of $\mathcal{F}_{2,k}$ 
that has the elements  
$1, \ldots, k$ in the first column and the elements $k+1, \ldots, 2k$ 
in the second column, then it is easy to see that 
$full_{n}^{P} =1$ for all $n \geq 1$ since the only 
element of $F \in \mathcal{F}_{n,k}$ with $\Pmch{F}=n-1$ 
has the entries $(i-1)k+1, \ldots, (i-1)k+k$ in the $i$-th column 
for $i =1, \ldots, n$.  More generally, if $\Upsilon$ is a subset of $\mathcal{F}_{2,k}$, then 
we define $Full_{n}^{\Upsilon}$ to be the set of 
$F \in \mathcal{F}_{n,k}$ such that $\Umch{F} = n-1$ 
and let $full_n^\Upsilon = |Full_n^\Upsilon|$.
Again, we use the convention that $full_{1}^{\Upsilon}=1$.
For 
example, if $\Upsilon = St_{2^k}$ is the set of all standard tableaux of shape 
$2^k$, then it easy to see that a $F \in \mathcal{F}_{n,k}$ has 
$\Umch{F} = n-1$ if and only if $F$ is standard tableaux of shape $n^k$, 
so 
$$full_n^{\Upsilon}=st_{n^k} = 
\frac{(kn)!}{\prod_{i=0}^{k-1} (i+n)\downarrow_n}.$$
 
If $\Upsilon$ is a subset of $\mathcal{F}_{2,k}$, then we shall be  
interested in the following three generating functions:  
\begin{eqnarray*}
D^{\Upsilon}(x,t) &=& 1+ \sum_{n \geq 1} \frac{t^n}{(kn)!} \sum_{F \in \mathcal{F}_{n,k}} 
x^{\Umch{F}}, \\
A^{\Upsilon}(t) &=& D^{\Upsilon}(0,t) = 1+ \sum_{n \geq 1} 
\frac{A^{\Upsilon}_{n,k}t^n}{(kn)!}, \ \mbox{and} \\
N^{\Upsilon}(x,t) &=& 1+ \sum_{n \geq 1} \frac{t^n}{(kn)!} \sum_{F \in \mathcal{F}_{n,k}} x^{\Ulap{F}} 
\end{eqnarray*}
where $A^{\Upsilon}_{n,k}$ is equal to the number of 
$F \in \mathcal{F}_{n,k}$ that have no $\Upsilon$-matches. 
If $\Upsilon$ consists of a single element $P$, then 
we write $D^P(x,t)$, $A^P(t)$, and $N^P(x,t)$ for 
$D^{\Upsilon}(x,t)$, $A^{\Upsilon}(t)$, and $N^{\Upsilon}(x,t)$, respectively. 
Clearly, the generating functions $D^{\Upsilon}(x,t)$ where 
$\Upsilon \subseteq St_{n^k}$ can be viewed as 
analogues of the generating function for rises as described 
in (\ref{risgf}).

We shall prove the following general theorems concerning the 
generating functions $D^{\Upsilon}(x,t)$, $A^{\Upsilon}(t)$, and 
$N^{\Upsilon}(x,t)$.

\begin{theorem}\label{thm:D} For all $\Upsilon \subseteq \mathcal{F}_{2,k}$, 
\begin{equation}\label{eq:D}
D^{\Upsilon}(x,t) = \frac{1-x}{1-x + 
\sum_{n \geq 1} \frac{((x-1)t)^n}{(kn)!} full^{\Upsilon}_{n}}.
\end{equation}
\end{theorem}

\begin{theorem}\label{thm:N} For all $\Upsilon \subseteq \mathcal{F}_{2,k}$, 
\begin{equation}\label{eq:N}
N^{\Upsilon}(x,t) = \frac{A^{\Upsilon}(t)}{1-x(1+(\frac{t}{k!}-1)A^{\Upsilon}(t))}.
\end{equation}
\end{theorem}

Theorem \ref{thm:D} is proved by applying a
ring homomorphism defined on the ring, $\Lambda$, of symmetric functions
over infinitely many variables,  $x_1,x_2, \ldots$,  to a
simple symmetric function identity. There has been a long line of research,
\cite{br}, \cite{b}, \cite{b2}, \cite{L},
\cite{LR}, \cite{MenRem1}, \cite{MenRem2}, \
\cite{Book}, \cite{RRW}, \cite{Wag}, 
that shows that a large number of generating functions for 
permutation statistics
can be obtained by applying homomorphisms defined on $\Lambda$ 
to simple symmetric function
identities. Theorem \ref{thm:N} is an 
analogue of a result of Kitaev \cite{K}.

By our remarks above, we have the following corollaries. 

\begin{corollary} Let $St_{2^k}$ denote the set of 
standard tableaux of shape $2^k$, then 

\begin{equation}\label{eq:DSt}
D^{St_{2^k}}(x,t) = \frac{1-x}{1-x +(x-1)\frac{t}{k!} +
\sum_{n \geq 2} \frac{((x-1)t)^n}{\prod_{i=1}^k (n+i-1)\downarrow_n}}
\end{equation}
and
\begin{equation}\label{eq:ASt}
A^{St_{2^k}}(t) = \frac{1}{1 -\frac{t}{k!} + 
\sum_{n \geq 2} \frac{(-t)^n}{\prod_{i=1}^k (n+i-1)\downarrow_n}}.
\end{equation}

\end{corollary}

\begin{corollary} Let $P$ be the element of $St_{2^k}$ that 
has $1, \ldots, k$ in the first column, then 

\begin{equation}\label{eq:DSSt}
D^{P}(x,t) = \frac{1-x}{1-x +(x-1)\frac{t}{k!} +
\sum_{n \geq 2} \frac{((x-1)t)^n}{(kn)!}}
\end{equation}
and
\begin{equation}\label{eq:ASSt}
A^{P}(t) = \frac{1}{1 -\frac{t}{k!}+ 
\sum_{n \geq 2} \frac{(-t)^n}{(kn)!}}.
\end{equation}

\end{corollary}

We have developed similar results for other 2-column patterns; 
the key is to be able to compute $full^P_{n}$. For example,  we 
can prove the following. 

\begin{theorem}

\begin{enumerate}
\item If $P_2^{(2,2)} = \begin{array}{|c|c|}
 \hline 3 & 4 \\
 \hline 1 & 2 \\
 \hline
 \end{array}$,
then $full^{P_2^{(2,2)}}_{n} =C_{n-1}$ where 
$C_n = \frac{1}{n+1}\binom{2n}{n}$ is the $n$-th Catalan number.

\item If 
$P_1^{(2,2,2)} = \begin{array}{|c|c|}
 \hline 4 & 6 \\ 
 \hline 2 & 5 \\
 \hline 1 & 3 \\
 \hline
 \end{array}$, 
then $full^{P_1^{(2,2,2)}}_{n} =\frac{1}{2n+1}\binom{3n}{n}$.
\end{enumerate}

\end{theorem}

Let $\mathbb{P}$ denote the set of positive integers $\{1,2, \ldots \}$, 
$\mathbb{O}= \{1,3,5,7, \ldots \}$ denote the set of odd numbers in $\mathbb{P}$, and 
$\mathbb{E}= \{2,4,6,8,\ldots \}$ denote the set of even numbers in $\mathbb{P}$. 
For $n\in\mathbb{P}$, 
let $[n] = \{1, \ldots, n\}$, $\mathbb{O}_n = [n] \cap \mathbb{O}$, and 
$\mathbb{E}_n = [n] \cap \mathbb{E}$. Now suppose that $\Upsilon$ is a subset 
of $\mathcal{F}_{2,k}$ and $F \in \mathcal{F}_{n,k}$. Then 
we let 
\begin{enumerate}
\item $S^{\Upsilon}(F)$ denote the set of $i$ such that 
there is an $\Upsilon$-match of $F$ starting at position $i$, 
\item  $\mathcal{F}^{(2),\Upsilon}_{n,k}$ denote the set of $F \in \mathcal{F}_{n,k}$ such that 
$\mathbb{O}_n \subseteq S^{\Upsilon}(F)$, and 
\item $\Upsilon\mbox{-mch}^{(2)}(F)$ denote the number of $i$ such that $2i \in S^{\Upsilon}(F)$. 
\end{enumerate}
We say that $F \in \mathcal{F}_{n,k}$ is {\bf $\Upsilon$-alternating} if 
$S^{\Upsilon}(F) = \mathbb{O}_n$. That is, $F \in \mathcal{F}_{n,k}$ is 
$\Upsilon$-alternating if $\red{F[2i+1,2i+2]} \in \Upsilon$ for all $0\leq i\leq (n-2)/2$ and 
$\red{F[2i,2i+1]} \not \in \Upsilon$ for all $0<i\leq (n-1)/2$. We let 
$Alt^{\Upsilon}_n$ denote the number of $\Upsilon$-alternating 
elements of $\mathcal{F}_{n,k}$. 

We shall prove the following general theorems about 
$\Upsilon$-alternating fillings. 

\begin{theorem}\label{thm:EAx} Let $\Upsilon$ be a subset of $\mathcal{F}_{2,k}$. Then 
\begin{equation*}
1+\sum_{n \geq 1} \frac{t^{2n}}{(2kn)!}\sum_{F \in \mathcal{F}^{(2)}_{2n,k}}
x^{\Upsilon\mbox{-mch}^{(2)}(F)} = \
\frac{1}{1- \sum_{n \geq 1} \frac{(x-1)^{n-1}t^{2n}}{(2kn)!}full^{\Upsilon}_{2n}}.
\end{equation*}
\end{theorem}

\begin{theorem}\label{thm:OAx} Let $\Upsilon$ be a subset of $\mathcal{F}_{2,k}$. Then 
\begin{equation*}
\sum_{n \geq 1} \frac{t^{2n-1}}{(k(2n-1))!}\sum_{F \in 
\mathcal{F}^{(2)}_{2n-1,k}}
x^{\Upsilon\mbox{-mch}^{(2)}(F)} = 
\frac{\sum_{n\geq 1} \frac{(x-1)^{n-1}t^{2n-1}}{(k(2n-1))!} full^{\Upsilon}_{2n-1}}{1- \sum_{n \geq 1} 
\frac{(x-1)^{n-1}t^{2n}}{(2kn)!}
full^{\Upsilon}_{2n}}.
\end{equation*}
\end{theorem}

Theorems \ref{thm:EAx} and \ref{thm:OAx} are also 
proved by applying a
ring homomorphism  defined on $\Lambda$ to 
simple symmetric function identities. 
Putting $x=0$ in Theorems \ref{thm:EAx} and \ref{thm:OAx}, we obtain 
the following generating functions for $\Upsilon$-alternating 
elements of $\mathcal{F}_{n,k}$ which can be viewed as generalizations 
of (\ref{alte}) and (\ref{alto}).

\begin{corollary}

\begin{equation}
1+\sum_{n \geq 1} \frac{Alt^{\Upsilon}_{2n}t^{2n}}{(2kn)!} =
 \frac{1}{1+ \sum_{n \geq 1} \frac{(-1)^{n}t^{2n}}{(2kn)!}full^{\Upsilon}_{2n}} \ \mbox{and}
\end{equation}

\begin{equation}
\sum_{n \geq 1} \frac{Alt^{\Upsilon}_{2n-1}t^{2n-1}}{(k(2n-1))!} =
\frac{\sum_{n\geq 1} \frac{(-1)^{n-1}t^{2n-1}}{(k(2n-1))!} full^{\Upsilon}_{2n-1}}{1+ \sum_{n \geq 1} 
\frac{(-1)^{n}t^{2n}}{(2kn)!}full^{\Upsilon}_{2n}}.
\end{equation}

 \end{corollary}

The outline of this paper is the following. In Section 2, 
we shall review the necessary background on symmetric functions 
needed to prove Theorems \ref{thm:D}, \ref{thm:EAx}, and \ref{thm:OAx}. 
Then in section 3, we shall give the proofs of Theorems \ref{thm:D}, \ref{thm:EAx}, and \ref{thm:OAx}. In section 4, we shall give the proof of 
Theorem \ref{thm:N}. 
In Section 5, we shall find formulas for 
$full_n^\Upsilon$ where $\Upsilon$ are subsets 
of the standard tableaux of shape $2^2$ and 
we shall completely classify all standard tableaux, $P\in St_{2^k}$, such 
that $full_n^P =1$ for all $n \geq 1$.
In addition, we shall briefly outline some methods that we have used to compute $full_n^P$ for certain 
standard tableaux $P \in St_{2^k}$ where $k \geq 3$.

\section{Symmetric Functions}
\label{section:sf}

In this section we give the necessary background on symmetric
functions needed for our proofs of Theorems \ref{thm:D}, \ref{thm:EAx}, and \ref{thm:OAx}.
We shall consider the ring of symmetric functions, $\Lambda$, over 
infinitely many variables $x_1, x_2, \ldots $.  The homogeneous 
symmetric functions, $h_n \in \Lambda$, and elementary symmetric functions, 
$e_n \in \Lambda$, are defined by the generating 
functions 
\begin{equation*}
H(t) = \sum_{n \geq 0} h_n t^n = \prod_{i=1}^\infty \frac{1}{1-x_it} 
\ \mbox{and} \ 
E(t) = \sum_{n \geq 0} e_n t^n = \prod_{i=1}^\infty (1+x_it).
\end{equation*} 
The $n$-th power symmetric function, $p_n \in \Lambda$, is defined as  
$$p_n = \sum_{i=1}^\infty x_i^n.$$ 

Let $\la = (\la_1,\dots,\la_\ell)$ be an integer partition; that is,
$\la$ is a finite sequence of weakly increasing non-negative
integers.  Let $\ell(\la)$ denote the number of nonzero integers in
$\la$. If the sum of these integers is $n$, we say that $\la$ is a
partition of $n$ and write $\la \vdash n$.  For any partition $\la =
(\la_1,\dots,\la_\ell)$, define $h_\la = h_{\la_1} \cdots
h_{\la_\ell}$, $e_\la = e_{\la_1} \cdots
e_{\la_\ell}$, and $p_\la = p_{\la_1} \cdots
p_{\la_\ell}$. The well-known fundamental theorem of symmetric
functions, see \cite{Mac}, says that 
$\{e_\la : \la \vdash n\}$ is a
basis for $\Lambda_n$, the space of symmetric functions 
which are homogeneous of degree $n$. Equivalently, the fundamental 
theorem of symmetric functions states that $\{e_0,e_1, \ldots \}$ 
is an algebraically independent set of generators for 
the ring $\Lambda$. It follows that one can completely specify a  
ring homomorphism $\Gamma: \Lambda \rightarrow R$ from 
$\Lambda$ into a ring $R$ by giving the values of $\Gamma(e_n)$ for 
$n \geq 0$.   

Next we give combinatorial interpretations 
to the expansion of $h_\mu$ in terms of the elementary symmetric 
functions.  Given partitions $\lambda =(\la_1, \ldots, \la_\ell)\vdash n$ and $\mu\vdash n$,
a $\lambda$-brick tabloid of shape $\mu$ is a filling 
of the Ferrers diagram of shape $\mu$ with bricks of size $\lambda_1, \ldots, \lambda_\ell$ such that 
each brick lies in single row and no two bricks overlap. For 
example, Figure  \ref{figure:labrick} shows all the 
$\la$-brick tabloids of shape $\mu$ where $\la = (1,1,2,2)$ and $\mu = (4,2)$.

\fig{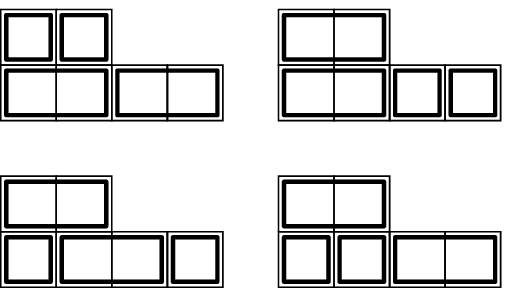}{The four $(1,1,2,2)$-brick tabloids of shape $(4,2)$.}

Let $\mathcal{B}_{\la,\mu}$ denote the set of all $\la$-brick tabloids of shape $\mu$ and let
$B_{\la,\mu} =|\mathcal{B}_{\la,\mu}|$.  E\u{g}ecio\u{g}lu and Remmel proved in
\cite{ER} that
\begin{equation}
\label{htoe}
h_\mu = \sum_{\la \vdash n} (-1)^{n - \ell(\la)} B_{\la,\mu} e_\la.
\end{equation}

If $T$ is a brick tabloid of shape $(n)$ such that the 
lengths of the bricks, reading from left to right, are 
$b_1, \ldots, b_\ell$, then we shall write 
$T=(b_1, \ldots, b_\ell)$. For example, 
the brick tabloid $T =(2,3,1,4,2)$ is pictured in 
Figure \ref{figure:btab}. 

\fig{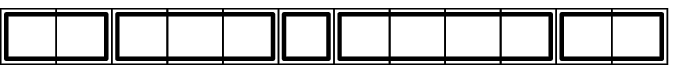}{The brick tabloid $T =(2,3,1,4,2)$.}

Next we define a 
 symmetric function $p_{n,\nu}$ whose relationship with $e_\lambda$,
which is very similar to the relationship between $h_n$ and
$e_\lambda$, was first introduced in \cite{LR} and
\cite{MenRem2}. Let $\nu$ be a function which maps the set of
non-negative integers into a field $F$. Recursively define
$p_{n,\nu} \in \La_n$ by setting $p_{0,\nu} = 1$ and letting
\begin{equation*}
p_{n,\nu}
= (-1)^{n-1} \nu(n) e_n + \sum_{k=1}^{n-1}(-1)^{k-1}e_k p_{n-k,\nu}
\end{equation*}
for all $n \geq 1$.  By multiplying series, this means that
\begin{equation*}
\left(\sum_{n \geq 0}(-1)^n e_n t^n \right) \left(\sum_{n \geq 1}
p_{n,\nu} t^n \right)
= \sum_{n \geq 1} \left ( \sum_{k=0}^{n-1} p_{n-k,\nu} (-1)^k e_k
\right ) t^n
= \sum_{n \geq 1} (-1)^{n-1} \nu(n) e_n t^n,
\end{equation*}
where the last equality follows from the definition of $p_{n,\nu}$.
Therefore,
\begin{equation}\label{pvu1}
\sum_{n \geq 1} p_{n,\nu}t^n
= \frac{\sum_{n \geq 1} (-1)^{n-1} \nu(n) e_n t^n}
{\sum_{n \geq 0} (-1)^n e_n t^n}
\end{equation}
or, equivalently,
\begin{equation}
\label{pnu2}
1+ \sum_{n \geq 1} p_{n,\nu}t^n
= \frac{1+ \sum_{n \geq 1} (-1)^{n}(e_n - \nu(n) e_n) t^n}
{\sum_{n \geq 0} (-1)^n e_n t^n}.
\end{equation}
If $\nu(n) = 1$ for all $n \geq 1$, then \eqref{pnu2} becomes
\begin{equation*}
1 + \sum_{n \geq 1} p_{n,1} t^n
= \frac{1}{\sum_{n \geq 0} (-1)^n e_n t^n}
= 1 + \sum_{n \geq 1} h_n t^n
\end{equation*}
which implies $p_{n,1} = h_n$.  Other special cases for $\nu$ give
well-known generating functions. For example, if $\nu(n) = n$ for $n
\geq 1$, then $p_{n,\nu}$ is the power symmetric function $p_n$.    
By taking $\nu(n) = (-1)^k \chi(n \geq k+1)$ for some $k
\geq 1$, $p_{n,(-1)^k\chi(n \geq k+1)}$ is the Schur function
corresponding to the partition $(1^k,n)$. Here for any statement 
$A$, we let $\chi(A) =1$ if $A$ is true and $\chi(A) =0$ if $A$ is 
false.

This definition of $p_{n,\nu}$ is desirable because of its expansion in terms
of elementary symmetric functions.  The coefficient of $e_\la$ in $p_{n,\nu}$
has a nice combinatorial interpretation similar to that of the homogeneous
symmetric functions.  Suppose $T$ is a brick tabloid of shape $(n)$ and type
$\la$ and that the final brick in $T$ has length $\ell$.  Define the weight
of a brick tabloid $w_{\nu}(T)$ to be $\nu(\ell)$ and let
\begin{equation*}
w_{\nu}(B_{\la,(n)}) = \sum_{\substack{\text{$T$ is a brick tabloid} \\
\text{of shape $(n)$ and type $\la$}}} w_{\nu}(T).
\end{equation*}
When $\nu(n) = 1$ for $n \geq 1$, $B_{\la,(n)}$ and $w_{\nu}(B_{\la,(n)})$ are
the same.  Then Mendes and Remmel \cite{MenRem2} proved that 
\begin{equation*}
p_{n,\nu} = \sum_{\la \vdash n}(-1)^{n-\ell(\la)} w_{\nu}(B_{\lambda,(n)}) e_\la.
\end{equation*}

\section{The proofs of Theorems \ref{thm:D}, \ref{thm:EAx}, and \ref{thm:OAx}}

In this section, we shall give the proofs of  Theorems \ref{thm:D}, \ref{thm:EAx}, and \ref{thm:OAx}. 
We start with the proof of Theorem \ref{thm:D}.

\ \\
{\bf Proof of Theorem \ref{thm:D}.}

\ \\
Our goal is to prove that for all $\Upsilon \subseteq \mathcal{F}_{2,k}$, 
\begin{equation}\label{eq:Fris}
D^{\Upsilon}(x,t) = 1+ \sum_{n \geq 1} \frac{t^n}{(kn)!} \sum_{F \in \mathcal{F}_{n,k}} 
x^{\Umch{F}} = \frac{1-x}{1-x + 
\sum_{n \geq 1} \frac{((x-1)t)^n}{(kn)!} full^{\Upsilon}_{n}}.
\end{equation}

Define a ring homomorphism
$\Gamma:\Lambda \rightarrow \mathbb{Q}[x]$, where $\mathbb{Q}[x]$ is the
polynomial ring over the rationals, 
by setting $\Gamma(e_0) =1$ 
and
\begin{equation}
\Gamma(e_{n}) = \frac{(-1)^{n-1}}{(kn)!} full^{\Upsilon}_{n} (x-1)^{n-1}
\end{equation}
for $n \geq 1$. Then we claim that
\begin{equation}\label{ris1}
(kn)!\Gamma(h_n) = \sum_{F \in \mathcal{F}_{n,k}} 
x^{\Umch{F}}
\end{equation}
for all $n \geq 1$.
That is,
\begin{eqnarray}\label{ris2}
(kn)!\Gamma(h_{n}) &=& 
(kn)!\sum_{\mu\vdash n} (-1)^{n-\ell(\mu)}B_{\mu, (n)}\Gamma(e_{\mu}) \nonumber \\
&=& (kn)! \sum_{\mu \vdash n} (-1)^{n-\ell(\mu)} \sum_{(b_1,
\ldots, b_{\ell(\mu)}) \in \mathcal{B}_{\mu,(n)}}
\prod_{j=1}^{\ell(\mu)} \frac{(-1)^{b_j-1}}{(kb_j)!} 
full^{\Upsilon}_{b_j} (x-1)^{b_j-1}  \nonumber \\
&=& \ \sum_{\mu \vdash n} \sum_{(b_1, \ldots, b_{\ell(\mu)}) \in \mathcal{B}_{\mu,(n)}} 
\binom{kn}{kb_1,\ldots,kb_{\ell(\mu)}} \prod_{j=1}^{\ell(\mu)}  full^{\Upsilon}_{b_j}(x-1)^{b_j-1}.
\end{eqnarray}

Next we want to give a combinatorial interpretation to
(\ref{ris2}).  First we interpret  
$\binom{kn}{kb_1,\ldots,kb_{\ell(\mu)}}$ as the number of 
ways to pick a set partition of $\{1,\ldots, kn \}$ into 
sets $S_1, \ldots, S_{\ell(\mu)}$ where $|S_j|=kb_j$ for 
$j=1, \ldots, \ell(\mu)$. 
For each set $S_j$, we interpret $full^{\Upsilon}_{b_j}$ as 
the number of ways to arrange the numbers in $S_j$ into a 
$k \times b_j$ array such that there is an $\Upsilon$-match 
starting at positions $1, \ldots, b_j-1$. 
Finally, we interpret $\prod_{j=1}^{\ell(\mu)}
(x-1)^{b_j-1}$ as all ways of picking labels for all the columns of each
brick, except the final column, with either an $x$ or a $-1$. For
completeness, we label the final column of each brick with $1$. We
shall call all such objects created in this way \emph{filled labeled
brick tabloids} and let $\mathcal{H}_{n,k}^{\Upsilon}$ denote the set of all
filled labeled brick tabloids that arise in this way.  Thus a $C
\in \mathcal{H}_{n,k}^{\Upsilon}$ consists of a brick tabloid, $T$, 
a filling, $F \in \mathcal{F}_{n,k}$, and a labeling, $L$, of the columns of $T$
with elements from $\{x,1,-1\}$ such that
\begin{enumerate}
\item $F$ has an $\Upsilon$-match starting at each column that
is not a final column a brick,    
\item the final column of each
brick is labeled with $1$, and 
\item each column that is not a final
column of a brick is labeled with $x$ or $-1$.
\end{enumerate}
We then define the weight of $C$, $w(C)$, to be the product of all
the $x$ labels in $L$ and the sign of $C$, $sgn(C)$, to be
the product of all the $-1$ labels in $L$. For example,
if $n =12$, $k=2$, $\Upsilon =St_{2^2}$, 
and $T =(4,3,3,2)$, then Figure \ref{figure:fris1}
demonstrates such a composite object $C \in \mathcal{H}_{12,2}$ where,
$w(C) = x^5$ and $sgn(C) =-1$.

Thus
\begin{equation}\label{ris4}
(kn)!\Gamma(h_{n}) = \sum_{C \in \mathcal{H}_{n,k}^{\Upsilon}}
sgn(C) w(C).
\end{equation}

\fig{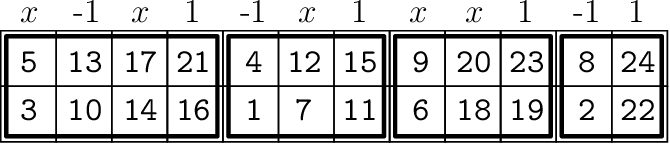}{A composite object $C \in \mathcal{H}_{12,2}^{St_{2^2}}$.}

Next we define a weight-preserving sign-reversing involution
$I:\mathcal{H}_{n,k}^{\Upsilon} \rightarrow \mathcal{H}_{n,k}^{\Upsilon}$.  
To define
$I(C)$, we scan the columns of $C =(T,F,L)$ from left  to right
looking for the leftmost column, $t$, such that either (i) $t$ is
labeled with $-1$ or (ii) $t$ is at the end of a brick, $b_j$, and
the brick immediately following $b_j$, namely $b_{j+1}$, has the property
that $F$ has an $\Upsilon$-match starting in each column of 
$b_j$ and $b_{j+1}$ except, possibly, the last cell of 
$b_{j+1}$.   In case (i), $I(C)
=(T',F',L')$ where $T'$ is the result of  replacing the brick
$b$ in $T$ containing $t$ by two bricks, $b^*$ and $b^{**}$, where
$b^*$ contains the $t$-th column plus all the columns in $b$ to the left
of $t$ and $b^{**}$ contains all the columns of $b$ to the right of
$t$, $F' =F$, and $L'$ is the labeling that results
from $L$ by changing the label of column $t$ from $-1$ to $1$. In
case (ii), $I(C) =(T',F',L')$ where $T'$ is the result of
replacing the bricks $b_j$ and $b_{j+1}$ in $T$ by a single brick
$b,$ $F' = F$, and $L'$ is the labeling that results
from $L$ by changing the label of column $t$ from $1$ to $-1$. If
neither case (i) or case (ii) applies, then we let $I(C) =C$. For
example, if $C$ is the element of $\mathcal{H}_{12,2}^{St_{2^2}}$ pictured in
Figure \ref{figure:fris1}, then $I(C)$ is pictured in Figure
\ref{figure:fris2}.

\fig{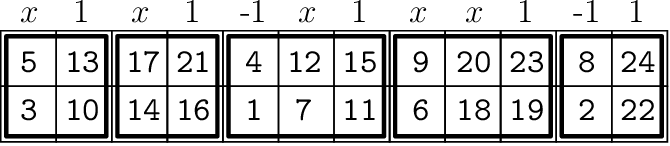}{$I(C)$ for $C$ in Figure \ref{figure:fris1}.}

It is easy to see that $I^2(C) =C$ for all 
$C \in \mathcal{H}_{n,k}^{\Upsilon}$ and that if 
$I(C) \neq C$, then $sgn(C)w(C) = -sgn(I(C))w(I(C))$. Hence  
$I$ is a weight-preserving and sign-reversing involution that shows 
\begin{equation}\label{ris5}
(kn)!\Gamma(h_n) = \sum_{C \in \mathcal{H}_{n,k}^{\Upsilon},I(C) = C}
sgn(C) w(C).
\end{equation}

Thus, we must examine the fixed points, $C = (T,F,L)$, of $I$.
First, there can be no $-1$ labels in $L$ so that $sgn(C) =1$.
Moreover,  if $b_j$ and $b_{j+1}$ are two consecutive bricks in $T$
and $t$ is the last column of $b_j$, then it cannot be the case that
there is an $\Upsilon$-match starting at position $t$ in $F$ 
since otherwise we
could combine $b_j$ and $b_{j+1}$. It follows that $sgn(C)w(C) = x^{\Umch{F}}$.
For example, Figure \ref{figure:fris3} shows a 
fixed point of $I$ in the case $n=12$, $k=2$, and 
$\Upsilon = St_{2^2}$. 

Vice versa, if
$F \in \mathcal{F}_{n,k}$, then we can create a fixed point, $C
=(T,F,L)$, by forming bricks in $T$ that end at columns that are not the 
start of an $\Upsilon$-match in $F$, labeling each column
that is the start of an $\Upsilon$-match in $F$ with $x$, and labeling the remaining columns 
with $1$. Thus we have shown that
\begin{equation*}
(kn)!\Gamma(h_n) = \sum_{F \in \mathcal{F}_{n,k}}x^{\Umch{F}}
\end{equation*}
as desired.

\fig{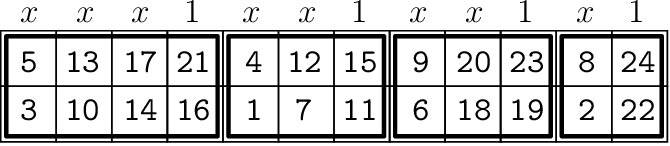}{A fixed point of $I$.}

Applying $\Gamma$ to the identity $H(t) = \frac{1}{E(-t)}$, we get
\begin{eqnarray*}
\sum_{n \geq 0} \Gamma(h_n) t^n &=& \sum_{n\geq 0} \frac{t^{n}}{(kn)!} \sum_{F \in \mathcal{F}_{n,k}}x^{\Umch{F}} \\
&=& \frac{1}{1+\sum_{n\geq 1} (-t)^n\Gamma(e_n)} \\
&=& \frac{1}{1+\sum_{n\geq 1}(-1)^{n} t^{n}
\frac{(-1)^{n-1}(x-1)^{n-1}}{(kn)!} full^{\Upsilon}_{n}}\\
&=& \frac{1-x}{1-x + \sum_{n\geq 1}
\frac{((x-1)t)^n}{(kn)!} full^{\Upsilon}_{n}}
\end{eqnarray*}
which proves (\ref{eq:Fris}).

Putting $x=0$ in (\ref{eq:Fris}) immediately yields the following 
corollary. 

\begin{corollary}  Let $\Upsilon \subseteq \mathcal{F}_{2,k}$ and 
$A^{\Upsilon}_{n,k}$ be the number of 
$F \in \mathcal{F}_{n,k}$ that have no $\Upsilon$-matches. Then 
\begin{equation} \label{eq:AUt}
A^{\Upsilon}(t) =  1+ \sum_{n \geq 1} 
\frac{A^{\Upsilon}_{n,k}t^n}{(kn)!}  = \frac{1}{1+\sum_{n\geq
1}\frac{(-t)^n}{(kn)!}full^{\Upsilon}_{n}}.
\end{equation}
\end{corollary}

Recall that if 
$\Upsilon \subseteq \mathcal{F}_{k,2}$, then 
$\mathcal{F}^{(2),\Upsilon}_{n,k}$ denotes the set of $F \in \mathcal{F}_{n,k}$ that have $\Upsilon$-matches starting  
at all the odd-numbered columns and that for 
$F \in \mathcal{F}^{(2),\Upsilon}_{n,k}$, 
$\Upsilon\mbox{-mch}^{(2)}(F)$ denotes the number of $i$ such that $F$ has 
an $\Upsilon$-match starting at position $2i$. 

Next we give the proof of Theorem \ref{thm:EAx}.

\ \\
{\bf Proof of Theorem \ref{thm:EAx}.}

\ \\
Our goal is to prove that if $\Upsilon$ is a subset of $\mathcal{F}_{2,k}$,
then 
\begin{eqnarray}\label{eq:2Fris}
D^{(2)}(x,t) &=& 1 +\sum_{n \geq 1} \frac{t^{2n}}{(2kn)!}\sum_{F \in \mathcal{F}^{(2)}_{2n,k}}
x^{\Upsilon\mbox{-mch}^{(2)}(F)} \nonumber \\
&=& \frac{1}{1- \sum_{n \geq 1} \frac{(x-1)^{n-1}t^{2n}}{(2kn)!}full^{\Upsilon}_{2n}}.
\end{eqnarray}

The proof of Theorem \ref{thm:EAx} is very similar to the proof of 
Theorem \ref{thm:D}. We only have to modify our 
ring homomorphism sightly. That is, 
define a ring homomorphism
$\Gamma^{(2)}:\Lambda \rightarrow \mathbb{Q}[x]$ by setting 
 $\Gamma^{(2)}(e_0) =1$, $\Gamma^{(2)}(e_{2n+1}) =0$ for all 
$n \geq 0$, and 
\begin{equation}
\Gamma(e_{2n}) = \frac{(-1)^{2n-1}}{(2kn)!} full^{\Upsilon}_{2n} (x-1)^{n-1}
\end{equation}
for $n \geq 1$.

Then we claim that for $n \geq 1$, 
\begin{equation}\label{2ris0}
\Gamma^{(2)}(h_{2n-1}) = 0
\end{equation}
and 
\begin{equation}\label{2ris1}
(2kn)!\Gamma^{(2)}(h_{2n}) = \sum_{F \in \mathcal{F}^{(2)}_{2n,k}} 
x^{\Upsilon\mbox{-mch}^{(2)}(F)}.
\end{equation}
Note that 
\begin{equation}\label{key}
\Gamma^{(2)}(h_{n}) =
\sum_{\mu\vdash n} (-1)^{n-\ell(\mu)}B_{\mu, (n)}\Gamma^{(2)}(e_{\mu}).
\end{equation}
Now if $n$ is odd, then every $\mu \vdash n$ must have an odd part 
and hence our definitions  
force $\Gamma^{(2)}(e_{\mu}) =0$. Thus 
$\Gamma^{(2)}(h_{2n-1}) = 0$ for all $n \geq 1$. On the other 
hand, if $n$ is even, then we can restrict ourselves 
to those partitions in (\ref{key}) that contain only even parts.
For any partition $\lambda =(\la_1, \ldots, \la_\ell)$ of $n$, we 
let $2\la$ denote the partition $(2\la_1, \ldots, 2\la_\ell)$. 
Then for $n \geq 1$ we have 

\begin{eqnarray}\label{2ris2}
(2kn)!\Gamma^{(2)}(h_{2n}) &=& 
(2kn)!\sum_{\mu\vdash n} (-1)^{2n-\ell(\mu)}B_{2\mu, (2n)}\Gamma^{(2)}(e_{2\mu}) \nonumber \\
&=& (2kn)! \sum_{\mu \vdash n} (-1)^{2n-\ell(\mu)} \sum_{(2b_1,
\ldots, 2b_{\ell(\mu)}) \in \mathcal{B}_{2\mu,(2n)}}
\prod_{j=1}^{\ell(\mu)} \frac{(-1)^{2b_j-1}}{(2kb_j)!} 
full^{\Upsilon}_{2b_j} (x-1)^{b_j-1}  \nonumber \\
&=& \ \sum_{\mu \vdash n} \sum_{(2b_1, \ldots, 2b_{\ell(\mu)}) \in \mathcal{B}_{2\mu,(2n)}} 
\binom{2kn}{2kb_1,\ldots,2kb_{\ell(\mu)}} \prod_{j=1}^{\ell(\mu)}  full^{\Upsilon}_{2b_j}(x-1)^{b_j-1}.
\end{eqnarray}

Again we can give a combinatorial interpretation to
(\ref{2ris2}).  First we interpret  
$\binom{2kn}{2kb_1,\ldots,2kb_{\ell(\mu)}}$ as the number of 
ways to pick a set partition of $\{1,\ldots, 2kn \}$ into 
sets $S_1, \ldots, S_{\ell(\mu)}$ 
where $|S_j|=2kb_j$ for $j=1, \ldots, \ell(\mu)$. 
For each set $S_j$, we interpret $full^{\Upsilon}_{2b_j}$ as 
the number of ways to arrange the numbers in $S_j$ into a 
$k \times 2b_j$ array such that there is an $\Upsilon$-match 
starting at positions $1, \ldots, 2b_j-1$ in the brick $2b_j$. 
Finally, we interpret $\prod_{j=1}^{\ell(\mu)}
(x-1)^{b_j-1}$ as all ways of picking labels 
for all the even numbered columns of each
brick, except the final even numbered column in the brick, with either an $x$ or a $-1$. For
completeness, we label the final even numbered 
column of each brick with $1$. We
shall call all such objects created in this way \emph{filled labeled
brick tabloids} and let $\mathcal{H}^{(2),\Upsilon}_{2n,k}$ 
denote the set of all
filled labeled brick tabloids that arise in this way.  Thus a $C
\in \mathcal{H}^{(2),\Upsilon}_{2n,k}$ consists of a brick tabloid, $T$,
a filling, $F \in \mathcal{F}_{2n,k}$, and a labeling, $L$, of the 
even numbered columns of $T$
with elements from $\{x,1,-1\}$ such that
\begin{enumerate}
\item all bricks have even size, 
\item $F$ has an $\Upsilon$-match starting at each column that 
is not a final column of its  brick,    
\item the final column of each
brick is labeled with $1$, and 
\item each even numbered column that is not a final
column of a brick is labeled with $x$ or $-1$.
\end{enumerate}
We then define the weight of $C$, $w(C)$, to be the product of all
the $x$ labels in $L$ and the sign of $C$, $sgn(C)$, to be
the product of all the $-1$ labels in $L$. For example,
if $n =12$, $k=2$, $\Upsilon =St_{2^2}$, 
and $T =(4,2,4,2)$, then Figure \ref{figure:2fris1}
pictures such a filled labeled brick tabloid 
$C \in \mathcal{H}^{(2),\Upsilon}_{12,2}$ where
$w(C) = x$ and $sgn(C) =-1$.

Thus
\begin{equation}\label{2ris4}
(2kn)!\Gamma^{(2)}(h_{2n}) = \sum_{C \in \mathcal{H}^{(2),\Upsilon}_{2n,k}}
sgn(C) w(C).
\end{equation}

\fig{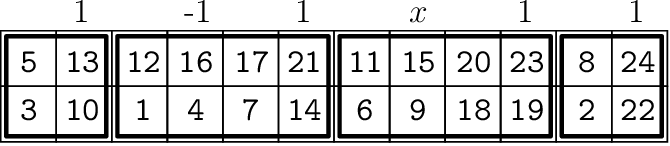}{A  $C \in \mathcal{H}^{(2),St_{2^2}}_{12,2}$.}

Next we define a weight-preserving sign-reversing involution
$I:\mathcal{H}^{(2),\Upsilon}_{2n,k} \rightarrow \mathcal{H}^{(2),\Upsilon}_{2n,k}$ essentially in the same way as in Theorem \ref{thm:D}.  
That is, to  define
$I(C)$, we scan the cells of $C =(T,F,L)$ from left  to right
looking for the leftmost column, $t$, such that either (i) $t$ is
labeled with $-1$ or (ii) $t$ is at the end of a brick, $2b_j$, and
the brick immediately following $2b_j$, namely $2b_{j+1}$, has the property
that $F$ has a $\Upsilon$-match starting in each column of
$2b_j$ and $2b_{j+1}$ except, possibly, the last cell of 
$2b_{j+1}$.   In case (i), $I(C)
=(T',F',L')$ where $T'$ is the result of  replacing the brick
$2b$ in $T$ containing $t$ by two bricks, $2b^*$ and $2b^{**}$, where
$2b^*$ contains the $t$-th column plus all the cells in $2b$ to the left
of $t$ and $2b^{**}$ contains all the columns of $2b$ to the right of
$t$, $F' =F$, and $L'$ is the labeling that results
from $L$ by changing the label of column $t$ from $-1$ to $1$. In
case (ii), $I(C) =(T',F',L')$ where $T'$ is the result of
replacing the bricks $2b_j$ and $2b_{j+1}$ in $T$ by a single brick
$2b,$ $F' = F$, and $L'$ is the labeling that results
from $L$ by changing the label of column $t$ from $1$ to $-1$. If
neither case (i) or case (ii) applies, then we let $I(C) =C$. For
example, if $C$ is the element of 
$\mathcal{H}^{(2),St_{2^2}}_{12,2}$ pictured in
Figure \ref{figure:2fris1}, then $I(C)$ is pictured in Figure
\ref{figure:2fris2}.

\fig{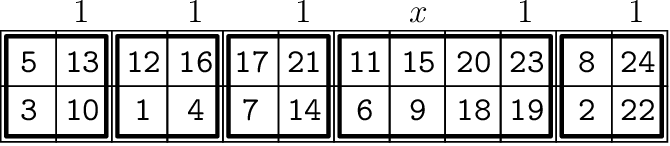}{$I(C)$ for $C$ in Figure \ref{figure:2fris1}.}

It is easy to see that the involution $I$ is weight-preserving and sign-reversing
and hence shows that
\begin{equation}\label{2ris5}
(2kn)!\Gamma^{(2)}(h_{2n}) = \sum_{C \in \mathcal{H}^{(2),\Upsilon}_{n,k},I(C) = C}
sgn(C) w(C).
\end{equation}

Thus, we must examine the fixed points $C = (T,F,L)$ of $I$.
First, by construction, $F$ must have $\Upsilon$-matches starting 
at all odd positions. Hence $F \in \mathcal{F}^{(2)}_{2n,k}$. 
As before, there can be no $-1$ labels in $L$ so that  $sgn(C) =1$.
Moreover,  if $2b_j$ and $2b_{j+1}$ are two consecutive bricks in $T$
and $t$ is the last column of $2b_j$, then it cannot be the case 
that there is an $\Upsilon$-match starting at position $t$ in $F$ 
since otherwise we
could combine $2b_j$ and $2b_{j+1}$. It follows that $sgn(C)w(C) = 
x^{\Upsilon\mbox{-mch}^{(2)}(F)}$. 

Vice versa, if
$F \in \mathcal{F}^{(2)}_{2n,k}$, then we can create a fixed point, $C
=(T,F,L)$, by forming bricks in $T$ that end at even numbered columns that are not the
start of an $\Upsilon$-match in $F$, and labeling each even numbered column
that is the start of an $\Upsilon$-match in $F$ with $x$, and labeling the remaining even numbered columns
with $1$. Thus we have shown that
\begin{equation*}
(2kn)!\Gamma^{(2)}(h_{2n}) = \sum_{F \in \mathcal{F}^{(2)}_{2n,k}}x^{\Upsilon\mbox{-mch}^{(2)}(F)}
\end{equation*}
as desired.

Applying $\Gamma^{(2)}$ to the identity $H(t) = \frac{1}{E(-t)}$, we get
\begin{eqnarray*}
\sum_{n \geq 0} \Gamma(h_n) t^n &=& \sum_{n\geq 0} \frac{t^{2n}}{(2kn)!}
\sum_{F \in \mathcal{F}^{(2)}_{2n,k}}x^{\Upsilon\mbox{-mch}^{(2)}(F)} \\
&=& \frac{1}{1+\sum_{n\geq 1} (-t)^{2n}\Gamma^{(2)}(e_{2n})} \\
&=& \frac{1}{1+\sum_{n\geq 1} t^{2n}
\frac{(-1)^{2n-1}(x-1)^{n-1}}{(2kn)!} full^{\Upsilon}_{2n}}\\
&=& \frac{1}{1 - \sum_{n\geq
1}\frac{(x-1)^{n-1}t^{2n}}{(2kn)!}full^{\Upsilon}_{2n}}
\end{eqnarray*}
which proves (\ref{eq:2Fris}).

As noted in the introduction, putting $x=0$ in (\ref{eq:2Fris}) gives 
us the generating function for the number of $\Upsilon$-alternating elements 
of $\mathcal{F}_{2n,k}$. That is, if $\Upsilon$ is a subset of $\mathcal{F}_{2,k}$, then 
\begin{equation}\label{eq:EvAlt}
1+\sum_{n \geq 1} \frac{Alt^{\Upsilon}_{2n}t^{2n}}{(2kn)!} =
 \frac{1}{1+ \sum_{n \geq 1} \frac{(-1)^{n}t^{2n}}{(2kn)!}
full^{\Upsilon}_{2n}}.
\end{equation}

Next we give the proof of Theorem \ref{thm:OAx}.

\ \\
{\bf Proof of Theorem \ref{thm:OAx}.}

\ \\
Our goal is to prove that if $\Upsilon$ is  a subset of $\mathcal{F}_{2,k}$, 
then 
\begin{eqnarray}\label{eq:3Fris}
D^{(2)}(x,t) &=& \sum_{n \geq 0} \frac{t^{2n+1}}{((2n+1)k)!}\sum_{F \in \mathcal{F}^{(2)}_{2n+1,k}}
x^{\Upsilon\mbox{-mch}^{(2)}(F)} \nonumber \\
&=& \frac{\sum_{n \geq 1} \frac{(x-1)^{n-1}t^{2n-1}}{(k(2n-1))!}full^{\Upsilon}_{2n-1}}
{1- \sum_{n \geq 1} \frac{(x-1)^{n-1}t^{2n}}{(2kn)!}full^{\Upsilon}_{2n}}.
\end{eqnarray}

Let $\Gamma^{(2)}:\Lambda \rightarrow \mathbb{Q}[x]$ be the ring 
homomorphism defined in Theorem \ref{thm:EAx}. Define 
$\nu:\mathbb{P} \rightarrow\mathbb{Q}$ by setting  
$\nu(n) = 0$ if $n$ is odd and setting 
\begin{equation}\label{nudef}
\nu(2n) = \frac{full^{\Upsilon}_{2n-1}(2kn)!}{full^{\Upsilon}_{2n}(k(2n-1))!}
\end{equation}
for $n \geq 1$.
We have defined $\nu$ so that for $n \geq 1$, 
\begin{eqnarray}\label{etimesnu}
\nu(2n)\Gamma^{(2)}(e_{2n}) &=& 
\frac{full^{\Upsilon}_{2n-1}(2kn)!}{full^{\Upsilon}_{2n}(k(2n-1))!}
\frac{(-1)^{2n-1}(x-1)^{n-1}}{(2kn!)}full^{\Upsilon}_{2n} \nonumber \\
&=& \frac{(-1)^{2n-1}(x-1)^{n-1}}{(k(2n-1))!}full^{\Upsilon}_{2n-1}.
\end{eqnarray}

Then we claim that for $n \geq 0$, 
\begin{equation}\label{3ris0}
\Gamma^{(2)}(p_{2n+1,\nu}) = 0
\end{equation}
and 
\begin{equation}\label{3ris1}
(k(2n+1))!\Gamma^{(2)}(p_{2n+2,\nu}) = \sum_{F \in \mathcal{F}^{(2)}_{2n+1,k}} 
x^{\Upsilon\mbox{-mch}^{(2)}(F)}.
\end{equation}
Note that 
\begin{equation}\label{3key}
\Gamma^{(2)}(p_{n,\nu}) =
\sum_{\mu\vdash n} (-1)^{n-\ell(\mu)}w_\nu(B_{\mu, (n)})\Gamma^{(2)}(e_{\mu}).
\end{equation}
Now if $n$ is odd, then every $\mu \vdash n$ must have an odd part 
and hence our definitions  
force $\Gamma^{(2)}(e_{\mu}) =0$. Thus 
$\Gamma^{(2)}(p_{2n+1,\nu}) = 0$ for all $n \geq 0$. On the other 
hand, if $n$ is even, then we can restrict ourselves 
to those partitions which contain only even parts in (\ref{3key}).  
Thus for $n \geq 1$ we have  
\begin{eqnarray}\label{3ris2}
&&(k(2n+1))!\Gamma^{(2)}(p_{2n+2,\nu}) = \nonumber \\
&&(k(2n+1))!\sum_{\mu\vdash n+1} (-1)^{2n+2-\ell(\mu)}w_\nu(B_{2\mu, (2n+2)})
\Gamma^{(2)}(e_{2\mu}) = \nonumber \\
&& (k(2n+1))! \sum_{\mu \vdash n+1} (-1)^{2n+2-\ell(\mu)} \times \nonumber \\ 
&& \ \ \ \ \sum_{(2b_1,
\ldots, 2b_{\ell(\mu)}) \in \mathcal{B}_{2\mu,(2(n+1))}}  
\frac{full^{\Upsilon}_{2b_{\ell(\mu)}-1}(2kb_{\ell(\mu)})!}{full^{\Upsilon}_{2b_{\ell(\mu)}}
(k(2b_{\ell(\mu)}-1))!}
\prod_{j=1}^{\ell(\mu)} \frac{(-1)^{2b_j-1}}{(2kb_j)!} 
full^{\Upsilon}_{2b_j} (x-1)^{b_j-1} =  \nonumber \\
&& \ \sum_{\mu \vdash n+1} \sum_{(2b_1, \ldots, 2b_{\ell(\mu)}) \in \mathcal{B}_{2\mu,(2(n+1))}} 
\binom{k(2n+1)}{2kb_1,\ldots,2kb_{\ell(\mu)-1},
k(2b_{\ell(\mu)}-1)} \times \nonumber \\
&& \ \ \ \ (x-1)^{b_{\ell(\mu)}-1} 
full^{\Upsilon}_{2b_{\ell(\mu)}-1}
\prod_{j=1}^{\ell(\mu)-1}  full^{\Upsilon}_{2b_j}(x-1)^{b_j-1}.
\end{eqnarray}

As before, we must  give a combinatorial interpretation to
(\ref{3ris2}).  First we interpret  \\
$\binom{k(2n+1)}{2kb_1,\ldots,2kb_{\ell(\mu)-1},
k(2b_{\ell(\mu)}-1)}$ as the number of 
ways to pick a set partition of $\{1,\ldots, k(2n+1) \}$ into 
sets $S_1, \ldots, S_{\ell(\mu)}$ 
where $|S_j|=2kb_j$ for $j=1, \ldots, \ell(\mu)-1$ and 
$|S_{\ell(\mu)}|=k(2b_{\ell(\mu)}-1)$. 
For each set $S_j$ with $j < \ell(\mu)$, 
we interpret $full^{\Upsilon}_{2b_j}$ as 
the number of ways to arrange the numbers in $S_j$ into a 
$k \times 2b_j$ array such that there is an $\Upsilon$-match 
starting at positions $1, \ldots, 2b_j-1$. 
We interpret $full^{\Upsilon}_{2b_{\ell(\mu)}-1}$ as
arranging the numbers in $S_{\ell(\mu)}$  in a
$k \times (2b_{\ell(\mu)}-1)$ array such that there is an $\Upsilon$-match 
starting at positions $1, \ldots, 2b_{\ell(\mu)}-2$. In 
this case, we imagine that the array fills the first 
$2b_{\ell(\mu)}-1$ columns of the brick $2b_{\ell(\mu)}$ and the last 
column is left blank. 
Finally, we interpret $\prod_{j=1}^{\ell(\mu)}
(x-1)^{b_j-1}$ as all ways of picking labels of all the even numbered columns 
of each
brick, except the final even numbered column in the brick, with either an $x$ or a $-1$. For
completeness, we label the final even numbered column 
in each brick with $1$. We
call all such objects created in this way \emph{filled labeled
brick tabloids} and let $\mathcal{K}^{(2),\Upsilon}_{2n+2,k}$ 
denote the set of all 
filled labeled brick tabloids that arise in this way.  Thus, a $C
\in \mathcal{K}^{(2),\Upsilon}_{2n+2,k}$ consists of a brick tabloid, $T$,
a filling, $F \in \mathcal{F}_{2n+1,k}$, and a labeling, $L$, of the 
even numbered columns of $T$
with elements from $\{x,1,-1\}$ such that
\begin{enumerate}
\item all bricks have even size, 
\item for all but the final brick, $F$ has an $\Upsilon$-match starting 
at each column that is not the final column of its brick, 
\item for the final brick, the last column is empty and  
$F$ has an $\Upsilon$-match starting at each column that 
is not one of the  final two columns of that brick,   
\item the final column of each
brick is labeled with $1$, and 
\item each even numbered column that is not a final even numbered 
column of a brick is labeled with $x$ or $-1$.
\end{enumerate}
We then define the weight of $C$, $w(C)$, to be the product of all
the $x$ labels in $L$ and the sign of $C$, $sgn(C)$, to be
the product of all the $-1$ labels in $L$. For example,
if $n =12$, $k=2$, $\Upsilon =St_{2^2}$, 
and $T =(4,2,4,2)$, then Figure \ref{figure:3fris1}
pictures such a filled labeled brick tabloid  
$C \in \mathcal{K}^{(2),\Upsilon}_{14,2}$ with
$w(C) = x^2$ and $sgn(C) =-1$.

Thus
\begin{equation}\label{3ris4}
(k(2n+1))!\Gamma^{(2)}(p_{2n+2,\nu}) = 
\sum_{C \in \mathcal{K}^{(2),\Upsilon}_{2n+2,k}}
sgn(C) w(C).
\end{equation}

\fig{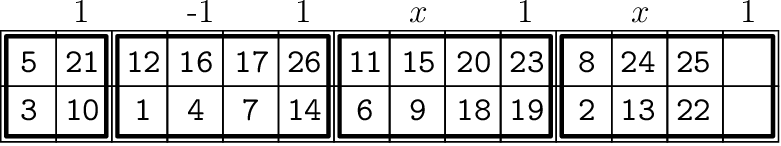}{A $C \in \mathcal{K}^{(2),St_{2^2}}_{14,2}$.}

At this point, we can follow the same set of steps as in 
theorem Theorem \ref{thm:EAx} to prove 
\begin{equation}\label{3fris8}
(k(2n+1))!\Gamma^{(2)}(p_{2n+2,\nu}) = \sum_{F \in \mathcal{F}^{(2)}_{2n+1,k}}x^{\Upsilon\mbox{-mch}^{(2)}(F)}.
\end{equation}
That is, we can define the weight-preserving sign-reversing involution
$I:\mathcal{K}^{(2),\Upsilon}_{2n+2,k} \rightarrow 
\mathcal{K}^{(2),\Upsilon}_{2n+2,k}$ in exactly the same way 
as in Theorem \ref{thm:EAx} and use it to prove (\ref{3fris8}) 
because the fact that the last cell is empty does not effect 
any of the arguments that we used in Theorem \ref{thm:EAx}. 

Applying $\Gamma^{(2)}$ to the identity (\ref{pvu1}) and 
using (\ref{etimesnu}) and (\ref{3fris8}), we obtain 

\begin{eqnarray*}
\sum_{n \geq 1} \Gamma^{(2)}(p_{n,\nu})t^n
&=& \sum_{n\geq 0} \frac{t^{2n+2}}{(k(2n+1)!} 
\sum_{F \in \mathcal{F}^{(2)}_{2n+1,k}}x^{\Upsilon\mbox{-mch}^{(2)}(F)} \\
&=& \frac{\sum_{n \geq 1} (-1)^{2n-1} \nu(2n) \Gamma^{(2)}(e_{2n}) t^{2n}}
{1+ \sum_{n \geq 0} (-1)^{2n} \Gamma^{(2)}(e_{2n}) t^{2n}} \\
&=& \frac{\sum_{n \geq 1} (-1)^{2n-1}\frac{(-1)^{2n-1}(x-1)^{n-1}}{(k(2n-1)!)}full^{\Upsilon}_{2n-1}t^{2n}}
{1 - \sum_{n\geq 1}\frac{(x-1)^{n-1}t^{2n}}{(2kn)!}full^{\Upsilon}_{2n}} \\
&=&
\frac{\sum_{n \geq 1} \frac{(x-1)^{n-1}}{(k(2n-1))!}full^{\Upsilon}_{2n-1}t^{2n}}{1 - \sum_{n\geq
1}\frac{(x-1)^{n-1}t^{2n}}{(2kn)!}full^{\Upsilon}_{2n}}.
\end{eqnarray*}
Dividing the above equation by $t$ will then yield 
\begin{equation*}
\sum_{n\geq 0} \frac{t^{2n+1}}{(k(2n+1)!} 
\sum_{F \in \mathcal{F}^{(2)}_{2n+1,k}}x^{\Upsilon\mbox{-mch}^{(2)}(F)}  =
\frac{\sum_{n \geq 1} \frac{(x-1)^{n-1}t^{2n-1}}{(k(2n-1))!}full^{\Upsilon}_{2n-1}}{1 - \sum_{n\geq
1}\frac{(x-1)^{n-1}t^{2n}}{(2kn)!}full^{\Upsilon}_{2n}} 
\end{equation*}
which proves (\ref{eq:3Fris}).

Putting $x=0$ in (\ref{eq:3Fris}) gives 
us the generating function for the number of $\Upsilon$-alternating elements 
of $\mathcal{F}_{2n+1,k}$. That is, if $\Upsilon$ be a subset of $\mathcal{F}_{2,k}$, then 
\begin{equation}\label{eq:EvAlt0}
\sum_{n \geq 0} \frac{Alt^{\Upsilon}_{2n+1}t^{2n+1}}{(k(2n+1)!)!} =
 \frac{\sum_{n \geq 1} \frac{(-1)^{n-1}t^{2n-1}}{(k(2n-1))!}full_{2n-1}}{1+ \sum_{n \geq 1} \frac{(-1)^{n}t^{2n}}{(2kn)!}
full^{\Upsilon}_{2n}}.
\end{equation}

\section{The proof of Theorem \ref{thm:N}.}

In this section we provide arguments similar to those
in~\cite[Sect. 4]{K} to prove Theorem \ref{thm:N}. 
That is, suppose that $\Upsilon \subseteq
\mathcal{F}_{j,k}$. Recall that
\begin{equation}\label{defNU1}
N_k^{\Upsilon}(t) = \sum_{n \geq 0} \frac{t^n}{(kn)!}
\sum_{F \in \mathcal{F}_{n,k}} 
x^{\Ulap{F}},
\end{equation}
and
\begin{equation}\label{AU1}
A_k^{\Upsilon}(t) = \sum_{n \geq 0} \frac{t^n}{(kn)!}A_{n,k}^{\Upsilon}
\end{equation}
where $A_{n,k}^{\Upsilon}$ is the number of 
$F \in \mathcal{F}_{n,k}$ with no $\Upsilon$-matches. 
Our goal is to prove that 
\begin{equation}\label{eq:N2}
N^{\Upsilon}(x,t) = \frac{A^{\Upsilon}(t)}{1-x(1+(\frac{t}{k!}-1)A^{\Upsilon}(t))}.
\end{equation}

Let $\mathcal{E}^{\Upsilon}_{n,k}$ denote the set of all
$F \in \mathcal{F}_{n,k}$ such that $F$ has exactly one
$\Upsilon$-match and it occurs at the end of $F$, i.e. 
the unique $\Upsilon$-match in $F$ starts
at position $n-j+1$. We then let 
$E_{n,k}^{\Upsilon} = |\mathcal{E}^{\Upsilon}_{n,k}|$ and 
\begin{equation}\label{Bdef}
B_k^{\Upsilon}(t) = \sum_{n \geq 1} \frac{E_{n,k}^{\Upsilon}t^n}{(kn)!}.
\end{equation}

\begin{lemma}\label{lem01} Suppose that $\Upsilon \subseteq
\mathcal{F}_{j,k}$. Then 
$B_k^{\Upsilon}(t)=1+ \left(\frac{t}{k!} -1\right)
A_k^{\Upsilon}(t)$.
\end{lemma}
\begin{proof}
Suppose that $F \in \mathcal{F}_{n,k}$.  For any set 
$S \subseteq \{1, \ldots, k(n+1)\}$ of size $k$, let $F^S$ be 
the $k \times (n+1)$ array so that the last column consists 
of the elements of $S$ in increasing order, reading from bottom to top,  
and the first $n$ columns is a $k \times n$ array, $G$, on 
the elements of $\{1, \ldots, k(n+1)\}-S$ such that $red(G) =F$. 
For example, if $S=\{4,6,9\}$  and $F \in \mathcal{F}_{4,3}$ pictured 
on the left in Figure \ref{figure:FS}, then $F^S$ is pictured on 
the right.

\fig{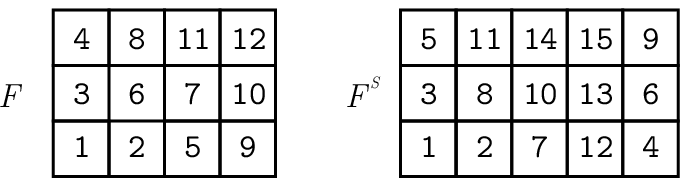}{An example of $F^S$.}

It is easy to see that for any set 
$S \subseteq \{1, \ldots, k(n+1)\}$ of size $k$, if $F$ has no $\Upsilon$-matches then either $F^S$ has 
no $\Upsilon$-matches or $F^S \in \mathcal{E}_{n+1,k}^{\Upsilon}$. 
It follows that
\begin{equation}\label{A=A+B}
\binom{k(n+1)}{k} A_{n,k}^{\Upsilon} = A_{n+1,k}^{\Upsilon} + 
E_{n+1,k}^{\Upsilon}
\end{equation}
If we multiply both sides of (\ref{A=A+B}) by $\frac{t^{n+1}}{(k(n+1))!}$ and
sum for $n \geq 0$, we get that
\begin{equation*}
\frac{t}{k!} A_k^{\Upsilon}(t) = A_k^{\Upsilon}(t)-1 + B_k^{\Upsilon}(t)
\end{equation*}
or that
\begin{equation*}
B_k^{\Upsilon}(t) = 1+ \left( \frac{t}{k!} -1\right) A_k^{\Upsilon}(t).
\end{equation*}
\end{proof}

Suppose that $F \in \mathcal{F}_{n,k}$ and 
$\Ulap{F}=i\geq 0$. One can read any such
$F$ from left to right making a cut right after
the first $\Upsilon$-match. Say the first $\Upsilon$-match
ends in column $i_1$. Then starting at column $i_1+1$ one 
can continuing reading left to right and make another 
cut after the first $\Upsilon$-match that one encounters. 
Continuing on in this way, 
one obtains $i$ fillings  $F_1, \ldots, F_i$ that have exactly one $\Upsilon$-match and
that $\Upsilon$-match occurs at the end of the filling  
and is possibly followed by a filling, $F_{i+1}$, with no $\Upsilon$-matches. In terms
of generating functions, this says that
\begin{eqnarray}\label{eq:N3}
N_k^{\Upsilon}(x,t) &=& A_k^{\Upsilon}(t) +
xB_k^{\Upsilon}(t) A_k^{\Upsilon}(t)+
(x B_k^{\Upsilon}(t))^2 A_k^{\Upsilon}(t)+\cdots  \nonumber \\
&=&\frac{A_k^{\Upsilon}(t)}{1-x B_k^{\Upsilon}(t)}.
\end{eqnarray}
Thus (\ref{eq:N2}) immediately follows by substituting the expression 
for $B_k^{\Upsilon}(t)$ given in Lemma~\ref{lem01} into (\ref{eq:N3}).

\section{Computing $full_{n}^{\Upsilon}$}

In this section, we show how to compute 
$full_{n}^{\Upsilon}$ for various $\Upsilon \subseteq \mathcal{F}_{2,k}$.

In the case $k =2$, we can compute $full_{n}^{\Upsilon}$ for all 
$\Upsilon$ that are subsets of the set of 
standard tableaux of shape $2^2$. We clearly  
have only 3 choices for $\Upsilon$ in that case, namely, 
$\Upsilon_0 =\{P_1^{(2,2)}\}$ and $\Upsilon_1 = \{P_2^{(2,2)}\}$, and 
  $\Upsilon_2 = \{P_1^{(2,2)},P_2^{(2,2)}\}$, where $P_1^{(2,2)}$ and 
$P_2^{(2,2)}$ are pictured in Figure \ref{figure:2standard}.   

For $\Upsilon_0$, it is easy to see that 
$F \in Full_{n}^{\Upsilon_0} = Full_{n}^{P_1^{(2,2)}}$ if and only if 
$F(1,i) =2i-1$ and $F(2,i) =2i$ for 
all $i$ so that $full_{n}^{P_1^{(2,2)}} =1$ for all $n \geq 2$.  
It is also easy to see that 
$F \in Full_{n}^{\Upsilon_2}$ if and only if $F$ is a standard 
tableau of shape $n^2$, so by the hook formula for the number 
of standard tableaux of shape $2^2$, 
$$ full_{n}^{\Upsilon_2}  = \frac{(2n)!}{n! (n+1)!} = 
\frac{1}{n+1} \binom{2n}{n} = C_n$$
where $C_n$ is $n$-th Catalan number. 

For $\Upsilon_1$, we have the following theorem.

\begin{theorem}
For all $n \geq 2$, $full_{n}^{P_2^{(2,2)}} = C_{n-1}$, 
where $C_n = \frac{1}{n+1} \binom{2n}{n}$ is the $n$-th Catalan number.
\end{theorem}
\begin{proof}
We shall give a bijective proof of this theorem by 
constructing a bijection from $\mathcal{D}_{n-1}$, which is the set 
of Dyck paths of length $2(n-1)$, onto $Full_n^{P_2^{(2,2)}}$. 
A Dyck path  of length $2n$ is a path that starts at (0,0) and ends at (2n,0) 
and consists of either {\em up-steps} (1,1) or {\em down-steps} (1,-1) in such a way that the 
path never goes below the $x$-axis.  It is well known that 
the number of Dyck paths of length $2n$ is equal to $n$-th Catalan 
number, $C_n = \frac{1}{n+1}\binom{2n}{n}$.

Now suppose that $F \in Full_n^{P_2^{(2,2)}}$.
Because there is a $P_2^{(2,2)}$-match starting in column $i$
for $i=1, \ldots, n-1$, it follows that $F(1,i)<F(1,i+1)<F(2,i)<F(2,i+1)$
for $i=1, \ldots, n-1$. Thus 
$F(1,1) < \cdots < F(1,n)$ and $F(2,1) < \cdots < F(2,n)$. Now since
$F(1,1) < F(1,2) < F(2,1)< F(2,2)$ and $F(n-1,1) <F(n,1)<F(2,n-1) < F(2,n)$, it 
follows that we must have that $F(1,1)=1$, $F(1,2)=2$, $F(2,n-1) = 2n-1$, and $F(2,n) =2n$. 

Let $\mathcal{R}_n$ be the set of fillings, 
$F$, of the $2 \times n$ rectangle with the integers $1, \ldots, 2n$ 
such that if the elements 
of $F$ in the first row, reading from left to right, are $a_1, \ldots a_n$,
and the elements 
of $F$ in the second row, reading from left to right, are $b_1, \ldots b_n$,
then (i) $a_1 < \cdots < a_n$, (ii) $a_1 =1$ and 
$a_2 =2$, (iii)  $b_1 < \cdots < b_n$, and (iv) $b_{n-1} = 2n-1$ and $b_n =2n$. 
Let $\mathcal{E}_{n-1}$ be the  set 
of paths of length $2(n-1)$ 
that consist of either {\em up-steps} (1,1) or {\em down-steps} (1,-1) such that the 
first step goes from $(0,0)$ to $(0,1)$ and the last step 
goes from $(2n-3,1)$ to $(2n-2,0)$.  Here we do not require 
that the paths in $\mathcal{E}_{n-1}$ stay above the $x$-axis.

There is a simple bijection, $\Theta$, which maps  
$\mathcal{E}_{n-1}$ onto $\mathcal{R}_n$.
Given a path $P = (p_1, \ldots, p_{2n-2}) \in \mathcal{E}_{n-1}$, we label the segments $p_2, \ldots,p_{2n-3}$ with the numbers 
$3, \ldots, 2n-2$, as pictured in 
Figure \ref{figure:P22bij}.  We then create a filling, 
$F = \Theta(P) \in \mathcal{R}_n$, by processing the elements  
$i \in \{3, \ldots, 2n-3\}$ in order by 
putting $i$ in the first available cell in the first row of $F$ if $p_{i-1}$ is an up-step and 
putting $i$ in the first available cell in the second row of $F$ if $p_{i-1}$ is an down-step.
See Figure \ref{figure:P22bij} for an example.

\fig{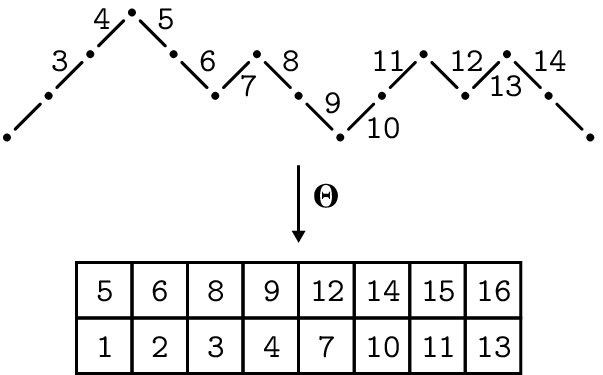}{An example of the map $\Theta$.} 

It is easy to see that $\Theta(P)$ will be an element of 
$\mathcal{R}_n$ since there are always $n-2$ up steps and $n-2$ down 
steps in $p_2, \ldots, p_{2n-3}$. Clearly, the process is reversible, so 
$\Theta$ is a bijection.  Thus all we have to prove 
is that $P \in \mathcal{D}_{n-1}$ if and only if $\Theta(P) \in 
Full_{n}^{P_2^{(2,2)}}$. Now suppose that $P = (p_1, \ldots, p_{2n-2}) \in 
\mathcal{D}_{n-1}$ and $F = \Theta(P)$.  We know that 
$F(1,1) = 1$, $F(1,2) =2$ and $2 < F(2,1) < F(2,2)$. Thus there 
is a $P_2^{(2,2)}$-match that starts in the first column of $F$. Now fix 
$2\leq j \leq n-1$. We shall show that there is a $P_2^{(2,2)}$-match starting in 
column $j$ of $F=\Theta(P)$. Suppose that $p_i$ is the $j$-th up-step 
of $P$, reading from left to right. Thus, $F(1,j+1)=i$ and the right-hand endpoint 
of $p_i$ must have height at least $1$, so there can be at most $j-1$ down-steps among 
$p_1, \ldots, p_i$. Thus, no more than $j-1$ cells of 
the second row of $F$ are filled with numbers less than $i$, while the 
first $j+1$ cells of the first row of $F$ are filled with numbers that 
are less than or equal to $i$. This means that 
$F(2,j-1)<F(1,j+1)=i< F(2,j)$ which implies that $F(1,j)<F(1,j+1)<F(2,j)<F(2,j+1)$. 
Hence there is a $P_2^{(2,2)}$-match starting in 
column $j$ of $F$. Thus $\Theta(P) \in Full_n^{P_2^{(2,2)}}$. 

Next suppose that $P = (p_1, \ldots, p_{2n-2}) \in \mathcal{E}_{n-1} - \mathcal{D}_{n-1}$. 
Then $P$ does not stay above the $x$-axis, so there 
is a smallest $i$ such that the right-hand endpoint of $p_i$ is
at level  $-1$.  It follows that for some $k \geq 2$, there must 
be $k$ down-steps and $k-1$ up-steps in $p_1, \ldots, p_i$,
so $i = 2k-1$. This means $F(2,k)=i+1=2k$ and the elements $1, \ldots, 2k$ occupy the 
first $k$ columns of $F= \Theta(P)$.  But this means 
that $F(1,k+1)>F(2,k)$, so there is not 
a $P_2^{(2,2)}$-match starting in column $k$ of $F$.  

Thus the map $\Theta$ restricts to a bijection between 
$\mathcal{D}_{n-1}$ and $Full_n^{P_2^{(2,2)}}$. Hence 
$full_n^{P_2^{(2,2)}} = C_{n-1}$ as claimed.

\end{proof}

Our next goal is to classify 
those standard tableaux, $P\in St_{2^k}$, such 
that $full_n^P=1$.  To this end, we say that a standard
tableau, $P\in St_{2^k}$, is 
{\em degenerate} if at least one of  
$P(i,1)+1 = P(i+1,1)$ or $P(i,2)+1 = P(i+1,2)$ holds for 
each $1 \leq i < k$. 
Then we have the following theorem. 

\begin{theorem}\label{degenerate}
Suppose $P\in St_{2^k}$ is a standard tableau 
where $k \geq 2$. Then 
$full_n^P =1$ for all $n \geq 1$ if and only if $P$ is degenerate.
\end{theorem}
\begin{proof}
Fix $P\in St_{2^k}$. To help us visualize the order relationships 
within $P$, we form a directed graph $G_P$ on the cells of the
$k \times 2$ rectangle by drawing a directed edge 
from the position of the number $i$ to the position 
of the number $i+1$ in $P$ for $i =1, \ldots, 2n-1$.  For example, in 
Figure \ref{figure:unique}, 
the $P\in St_{2^5}$ pictured on the left, results 
in the directed graph $G_P$ pictured 
immediately to its right. Then $G_P$ determines the order relationships between all the cells in $P$
since $P(r,s)<P(u,v)$ if there is a directed 
path from cell $(r,s)$ to cell $(u,v)$ in $G_P$. 
Now suppose that $F \in Full_n^P$ where $n \geq 3$. Because  
there is a $P$-match starting in column $i$, we can superimpose 
$G_P$ on the cells in columns $i$ and $i+1$ to determine the order 
relations between the elements in the those two columns. If we do 
this for every pair of columns, $i$ and $i+1$ for $i =1, \ldots, n-1$, 
we end up with a directed graph on the cells of 
the $k \times n$ rectangle which we will call $G_{P,n}$. 
For example, Figure \ref{figure:unique2},
$G_{P,4}$ is pictured just to the right of $G_P$.  It is then 
easy to see that if $F \in Full_n^P$ and there is a directed 
path from cell $(r,s)$ to cell $(u,v)$ in $G_{P,n}$, then 
it must be the case that $F(r,s) < F(u,v)$. Note that $G_{P,n}$ will always be a directed acyclic 
graph with no  multiple edges.

In general, the problem of computing $full_n^P$ for a column 
strict tableau $P$ of shape $2^k$ can be reduced to finding 
the number of linear extensions of a certain poset. 
That is, we claim the graph $G_{P,n}$ induces a poset 
$\mathcal{W}_{P,n} = (\{(i,j): 1 \leq i \leq k\ \& \ 1 \leq j \leq n\}, <_W)$ 
on the cells of the 
$k \times n$ rectangle by defining $(i,j) <_W (s,t)$ if 
and only if there is path from $(i,j)$ to $(s,t)$ in $G_{P,n}$. 
It is easy to see that $<_W$ is transitive. Since $G_{P,n}$ 
is acyclic, we can not have $(r,s) <_W (r,s)$. Thus 
$\mathcal{W}_P(n)$ is poset. It follows that there 
is a 1:1 correspondence between the elements of $Full_n^P$ 
and the linear extensions of $\mathcal{W}_{P,n}$. That is, 
if $F \in Full_n^P$, then it is easy to see 
that $(a_1,b_1), \ldots, (a_{kn},b_{kn})$ where $F(a_i,b_i)= i$ is 
a linear extension of $\mathcal{W}_{P,n}$. Vice versa, if 
$(a_1,b_1), \ldots, (a_{kn},b_{kn})$ is a linear extension of 
$\mathcal{W}_{P,n}$, then we can define 
$F$ so that $F(a_i,b_i)=i$ and it will automatically be the case 
that $F \in Full_n^P$. Since every poset has at least one 
linear extension, it follows that $Full_n^P \neq \emptyset$ 
for all $n \geq 2$. Thus $full_n^P \geq 1$ for all $n \geq 1$.

Our next goal is 
to apply a simple lemma on directed acyclic geraphs with no  multiple 
edges to replace $G_{P,n}$ by a simpler acyclic directed graph 
which contains the same information about the relative order of the 
elements in $F$.  Given  a directed acyclic graph $G = (V,E)$ 
with no multiple edges, let $Con(G)$ equal the set 
of all pairs $(i,j) \in V \times V$ such that there is a directed 
path in $G$ from vertex $i$ to vertex $j$.  Then we have the following.

\begin{lemma} \label{lem:remove}
Let $G = (V,E)$ be a directed acyclic graph 
with no multiple edges. Let $H$ be the subgraph of $G$ that 
results by removing all edges $e=(i,j) \in E$ such that 
there is a directed path from $i$ to $j$ in $G$ that does not 
involve $e$. Then $Con(G) = Con(H)$. 
\end{lemma}
\begin{proof}
The proof will be by induction on the number, $k$, of edges 
$e=(i,j) \in E$ such that 
there is a directed path from $i$ to $j$ in $G$ that does not 
involve $e$. If $k =0$, there is nothing to prove. 

If $k > 0$, let $e=(i,j) \in E$ be an edge such that 
there is a directed path, 
$$P = (i=v_1,v_2, \ldots, v_s =j),$$ 
from $i$ to $j$, in $G$, that does not 
involve $e$. Let $H^*$ be the directed graph that results from $G$ by 
removing edge $e$. Now suppose that $(x,y) \in Con(G)$. 
If there is a directed path from $x$ to $y$ in $G$ that does not involve 
$e$, then clearly $(x,y) \in Con(H^*)$. Otherwise, suppose that 
$(x=w_1, \ldots, w_t =y)$ is a directed path from $x$ to $y$ in $G$, for some $x,y\in V$, that 
uses the edge $e$. Thus for some $1\leq r <t$, we have $w_r =i$ and 
$w_{r+1} =j$. But then $(w_1, \ldots, w_r,v_2,\ldots,v_{s-1},w_{r+1}, 
\ldots, w_t)$ is a path in $H^*$ which connects $x$ and $y$ in $H^*$. 
Thus it follows that $H^*$ is a connected acyclic graph with 
no multiple edges such that $Con(G) = Con(H)$.  However, 
$H^*$ has $k-1$ edges, $f=(m,n) \in E-\{e\}$, such that 
there is a directed path from $m$ to $n$ in $H^*$ that does not 
involve $f$. Thus by induction, we can remove all such edges 
from $H^*$ to obtain an $H$ with $Con(G) = Con(H^*) = Con(H)$.
\end{proof}

Now suppose that we apply Lemma \ref{lem:remove} to the graph 
$G_{P,4}$ in Figure \ref{figure:unique} to produce $H_{P,4}$. We know that 
for any $((r,s),(u,v)) \in Con(G_{P,4})$,  it must be the case that $F(r,s) < F(u,v)$ 
for all $F \in Full_4^P$. Thus since $Con(G_{P,4})=Con(H_{P,4})$, it 
must be that case that whenever $((r,s),(u,v)) \in Con(H_{P,4})$, 
$F(r,s) < F(u,v)$  for all $F \in Full_4^P$. Now the directed edge 
$((1,2),(2,2))$ is not in $H_{P,4}$ since $G_{P,4}$ has the directed path 
$((1,2),(4,1),(5,1),(2,2))$.  Also, 
directed edge $((3,2),(4,2))$ is not in $H_{P,4}$
 because of the directed path $((3,2),(1,3),(4,2))$. 
Similarly, the corresponding 
edges in the third column of $G_{P,4}$ will be removed when creating $H_{P,4}$, so 
$H_{P,4}$ is the directed graph on the cells of the $5 \times 4$ rectangle pictured 
on the far right of Figure \ref{figure:unique}.   Observe that in this case, 
$H_{P,4}$ has the property that outdegree of each vertex is 1  
which means that $H_{P,4}$ determines a total order on the entries of any $F \in Full_4^P$. Thus, there is exactly one $F \in Full_4^P$, which means that $full^P_4 =1$ in this case.

\fig{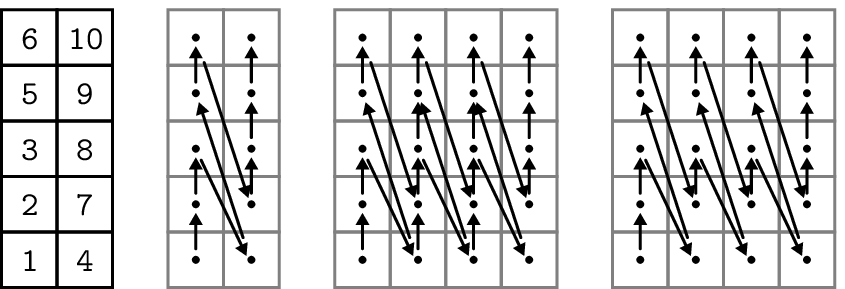}{$P$, $G_P$, $G_{P,4}$, and $H_{P,4}$.}

We should note, however, that it is not always be the case 
that applying Lemma \ref{lem:remove} to remove 
edges from $G_{P,n}$ will produce a graph, 
$H_{P,n}$, where each vertex has outdegree 1. For example, 
consider the standard tableau $Q$ pictured on the left in Figure 
\ref{figure:unique2}.  The graph $G_Q$ is pictured immediately 
to the right $Q$ and the graph $G_{Q,4}$ is pictured immediately 
to the right of $G_Q$.  The directed edge $((1,2),(2,2))$ is not 
in $H_{Q,4}$ due to the directed path of 
length 2, $((1,2),(4,1),(2,2))$, in $G_{Q,4}$.  Also, the directed 
edge $((4,2),(5,2))$ is not in $H_{Q,4}$  
since $G_{Q,4}$ has the directed path 
$((4,2),(2,3),(3,3),(5,2))$.  
Similarly, the corresponding 
edges in the third column of $G_{Q,4}$ are eliminated when creating $H_{Q,4}$ from 
$G_{Q,4}$.  It is then easy to check that  $H_{Q,4}$ is 
the directed 
graph on the cells of the $5 \times 4$ rectangle pictured 
on the far right of Figure \ref{figure:unique2}. Note that in 
this case, cells $(3,3)$ and $(4,3)$ have outdegree 2. In fact, 
it is not difficult to see that there are exactly four $F \in 
Full^Q_4$ so that $full^Q = 4$ in this case. 

\fig{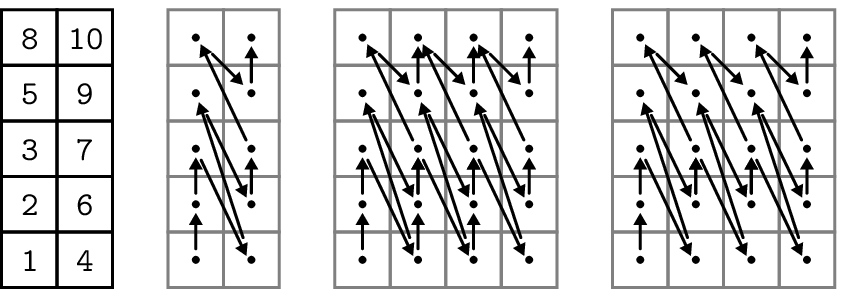}{$Q$, $G_Q$, $G_{Q,4}$, and $H_{Q,4}$.}

Let $P \in \mathcal{F}_{2,k}$ be a standard tableau. For all $n \geq 2$, 
let $H_{P,n}$ be the acyclic directed graph that arises 
from $G_{P,n}$ by removing all edges $e=((r,s),(u,v))$ in $G_{P,n}$ such that 
there is a directed path from $(r,s)$ to $(u,v)$ in $G_{P,n}$ that does 
not involve $e$. By Lemma \ref{lem:remove}, we know that 
$H_{P,n}$ imposes the same order relations on any 
$F \in Full^P_n$ as were imposed by $G_{P,n}$.

Suppose that $P$ is a degenerate standard tableau in 
$St_{2^k}$. We claim the outdegree of any vertex in $H_{P,n}$ is 1, which means that 
$H_{P,n}$ imposes a total order on the cells in the $k \times n$ rectangle. Thus, 
there is only one $F \in Full^P_n$ and, hence, $full_n^P =1$ 
for all $n \geq 1$. 
Let $E = E(G_P)$ be the set of directed edges in the 
graph $G_P$ and $E_n = E(G_{P,n})$ be the set of directed edges in 
$G_{P,n}$.  First observe the out degree of any cell 
in column 1 or column $n$ of $G_{P,n}$ is 1 since $G_P$ is superimposed 
only once in column 1 and column $n$. However, for 
any column $c$ with $2 \leq c \leq n-1$, $G_P$ is superimposed 
on columns $c-1$ and $c$ and on columns $c$ and $c+1$ so 
the outdegree of each cell in column $c$ could be 2. Next, observe that since 
$P$ is a standard tableau, we always have 
that $P(i,1) < P(i,2)$ and the elements in 
the columns of $P$ are strictly 
increasing from bottom to top. This means that there are three possibilities 
for an edge in $E$ coming out of cell $(i,1)$ for $i \leq n-1$, namely, it can 
go to cell $(i+1,1)$, to cell $(i,2)$, or to some cell $(g,2)$, where 
$g < i$.  Similarly there are only two possibilities 
for an edge in $E$ coming out of cell $(i,2)$ for $i \leq n-1$, namely, it can 
go to cell $(i+1,2)$ or it can go to cell $(j,1)$, where $j > i$. 
Since $P$ is degenerate, we must have at least 
one of the edges $((i,1),(i+1,1))$ and $((i,2),(i+1,2))$ in $G_P$, so 
in fact, we only have 4 possible cases for pairs of edges that leave cells 
$(i,1)$ and $(i,2)$, as pictured in Figure 
\ref{figure:possible}.

\fig{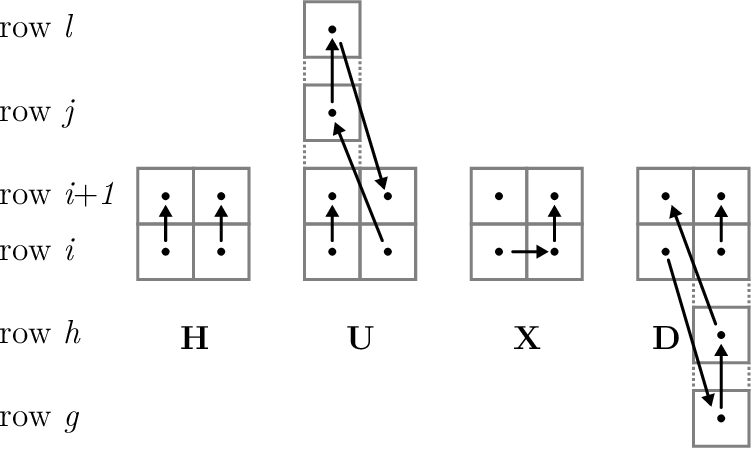}{Possible edges in $G_P$ for a degenerate standard tableau, $P$.}

\ \\
{\bf Case H.} Edges $((i,1),(i+1,1))$ and $((i,2),(i+1,2))$ are in $G_P$.\\
In this case, the outdegree of cell $(i,u)$ in column $u$ of $H_{P,n}$ is 1 because  $((i,u),(i+1,u))$ is the only edge out of each $(i,u)$ in $G_{P,n}$.

\ \\
{\bf Case U.}  Edges $((i,1),(i+1,1))$ and $((i,2),(j,1))$ are in $G_P$ for some $j>i+1$. \\
In this case, since $P(i,2)+1=P(j,1)$ and $P(i+1,2)>P(i,2)$, 
$G_P$ contains the path, $((i,2),(j,1),\ldots,(l,1),(i+1,2))$, for some $l\geq j$, 
which passes through all the cells in the first column from $(j,1)$ to $(l,1)$.
Now consider $G_{P,n}$, as shown in Figure \ref{figure:CaseB}. 
For each interior column, $2\leq u\leq n-1$, $G_{P,n}$ contains the paths 
$((i,u),(i+1,u))$ and $((i,u),(j,u-1),\cdots,(l,u-1),(i+1,u))$.
Therefore, $H_{P,n}$ omits the edge $((i,u),(i+1,u))$ 
in all but its first column, leaving the single edge, $((i,u),(j,u-1))$, 
exiting each cell in row $i$.

\fig{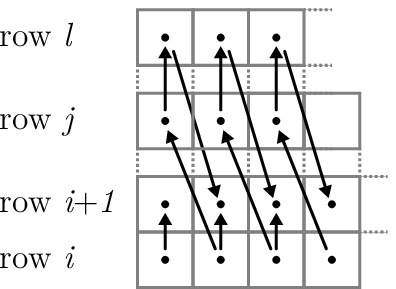}{Edges in $G_{P,n}$ in Case U.} 

\ \\
{\bf Case X.} Edges $((i,1),(i,2))$ and $((i,2),(i+1,2))$ are in $G_P$. \\
This case never occurs. Since $P$ increases along rows and columns, 
this case would require $P(i,1)<P(i+1,1)<P(i+1,2)$ while maintaining
$P(i,1)+2=P(i,2)+1=P(i+1,2)$, which does not have an integer 
solution unless $P(i+1,1)=P(i,2)$, which is impossible in $P$.

\ \\
{\bf Case D.} Edges $((i,1),(g,2))$ and $((i,2),(i+1,2))$ are in $G_P$ for some $g < i$. \\
In this case, since $P(i,1)+1=P(g,2)$ and $P(i+1,1)>P(i,1)$, 
$G_P$ contains the path, $((i,1),(g,2),\cdots,(h,2),(i+1,1))$, for some $g\leq h<i$, which passes 
through all the cells in the second column from $(g,2)$ to $(h,2)$.
Now consider $G_{P,n}$, as shown in Figure \ref{figure:CaseD}. 
For each interior column, $2\leq u\leq n-1$, $G_{P,n}$ contains the paths 
$((i,u),(i+1,u))$ and $((i,u),
(g,u+1),\cdots,(h,u+1),(i+1,u))$.
Therefore, $H_{P,n}$ omits the edge $((i,u),(i+1,u))$ in all but its
last column, leaving a single edge, $((i,u),(g,u+1))$, exiting each cell 
in row $i$.

\fig{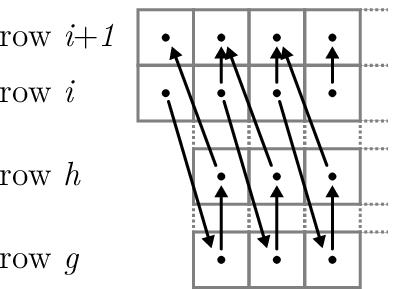}{Edges in $G_{P,n}$ in Case D.} 

\noindent
This shows that the outdegree of each vertex in $H_{P,n}$ is 1 as desired.

Now suppose that $P$ is a standard tableau of shape 
$2^k$ that is not degenerate.  We will show 
that $full_n^P \geq 2$.  Let $i< k$ be the smallest $i$
such that neither $P(i,1) + 1 = P(i+1,1)$ nor 
$P(i,2) + 1 = P(i+1,2)$.  Thus, $G_P$ has the edges 
$((i,1),(g,2))$ and $((i,2),(j,1))$, 
where $g\leq i<j$. 
Fix $n \geq 3$. We know $Full_n^P \neq \emptyset$ so that 
we fix some $F \in Full_n^P$.

Since $F$ is column strict, all the elements in the second column 
below cell $(i,2)$ are less than $y=F(i,2)$ and all the elements in the 
second column above cell $(i+1,2)$ are greater than $z=F(i+1,2)$.
 That is, 
we have the situation pictured in Figure \ref{figure:CaseII}.

\fig{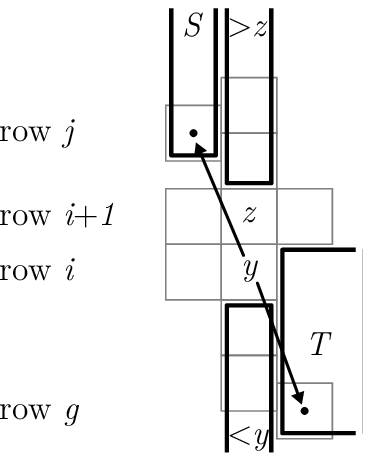}{$G_{P,n}$ when $P$ is not degenerate.}

Let $S$ be the elements in the first column of $F$ which 
are in the interval $(y,z)$ and let $T$ be the elements of $F$ in columns 
$3, \ldots, n$ which are in the interval $(y,z)$. Since 
$F$ has a $P$-match in the first two columns, we 
know that $F(j,1)$ must be an immediate successor of 
$F(i,2)=y$ relative to the elements of the first two columns of $F$ 
so that $y < F(j,2) <z$.  Similarly, since $F$ has a $P$-match in 
columns 2 and 3, we 
know that $F(g,3)$ must be an immediate successor of 
$F(i,2)=y$ relative to the elements of in columns 2 and 3 of $F$ 
so that $y < F(g,3) <z$. Thus we know that $S$ and $T$ are non-empty. 
Let $s =|S|$ and $t=|T|$. It follows that the interval 
$(y,z)$ consists of $s$ elements in 
the first column and $t$ elements in columns $3,\ldots,n$ of $F$.  
Next suppose that we pick any set of $S_1$ of $s$ elements of 
from $S \cup T$. We can then we modify the positions of the elements 
in the interval $(y,z)$ in $F$ by placing the elements 
of $S_1$ in column 1 in the positions that were occupied by $S$ in column 
1 of $F$ in the same relative order as the elements 
of $S$ in column 1 of $F$ and placing the elements 
of $T_1 = (S \cup T) - S_1$ in the positions that were occupied by $T$ in 
columns $3, \ldots, n$  of $F$ in the same relative order as the elements 
of $T$ in columns $3, \ldots, n$ of $F$. Call this modified filling $F_{S_1}$. 
Because we only changed the positions of elements in 
the interval $(y,z)$ in $F$, it is easy to see that 
$F_{S_1}$ will be an element of 
$Full_n^P$. 
Thus, $Full_n^P$ contains at least 
$\binom{s+t}{s}$ elements. Since $s,t \geq 1$, 
$\binom{s+t}{s} \geq 2$ and, hence,  $full_n^P \geq 2$ for  all $n \geq 3$.
\end{proof}

Now that we have classified the standard
tableaux, $P$, of shape $2^k$ such that 
$full_n^P =1$ for all $n \geq 1$, we will determine the number of
such column strict tableaux. Let $Deg_{2,k}$ be the set 
degenerate column strict tableaux of shape $2^k$. 
A Motzkin path 
of length $n$ is a lattice path on the integer lattice in the 
plane that runs from $(0,0)$ to $(n-1,0)$ consisting of 
up-steps $U =(1,1)$, down-steps $D =(1,-1)$, and horizontal-steps 
$H=(1,0)$ so that the path never passes below the $x$-axis,   
We let $\mathcal{M}_n$ 
denote the set of Motzkin paths of length $n$. If we are allowed 
to color the horizontal-steps of the paths $M \in \mathcal{M}_n$ 
 with one of $k$ colors, then 
we obtain the set of $k$-colored Motzkin paths of length $n$. We 
let $\mathcal{M}^{(k)}_n$ denote the set of $k$-colored 
Motzkin paths. For $\mathcal{M}^{(2)}_n$, we shall denote 
the 2-colored horizontal-steps as $H$ and $\widetilde{H}$.  
It was proved in \cite{DS}, \cite{ST} that 
$|\mathcal{M}^{(2)}_n| =C_{n+1}$  where $C_n$ is the $n$-th 
Catalan number. By the hook  
formula, we know that  $|St_{2^n}| = C_n$. Thus 
there exists a bijection from the set of standard
tableaux of shape $2^n$ onto the set of $2$-colored 
Motzkin paths of length $n-1$. We claim that we can 
define such a bijection $\Gamma: St_{2^n} 
\rightarrow \mathcal{M}^{(2)}_{n-1}$  for $n \geq 2$ as follows. 
Given a column strict tableau $P\in St_{2^n}$,  
define $\Gamma(P)$  to be the path 
$M = (m_1, \ldots, m_{n-1})$ such that $m_i$ equals 
$$\left\{\begin{array}{cl}
U & \textrm{if }G_P\textrm{ contains }(i,1)\rightarrow (i+1,1)\textrm{ and }(i,2)\rightarrow (j,1)\textrm{ for some } j > i+1, \\
D & \textrm{if }G_P\textrm{ contains }(i,1)\rightarrow (g,2)\textrm{ and }(i,2)\rightarrow (i+1,2)\textrm{ for some } g <i, \\
H & \textrm{if }G_P\textrm{ contains }(i,1)\rightarrow (i+1,1)\textrm{ and }(i,2)\rightarrow (i+1,2), \textrm{ and} \\
\widetilde{H} & \textrm{if }G_P\textrm{ contains }(i,1)\rightarrow (g,2)\textrm{ and }(i,2)\rightarrow (j,1)\textrm{ for some }g\leq i\textrm{ and }j>i.
\end{array}\right.$$

Note that we have not used the edge, $e$, out of cell $(n,1)$ in 
this definition. We must first check that $M$ is a 2-colored Motzkin path. 
This is equivalent to showing 
that (i) there are an equal number of $U$'s and $D$'s, i.e., the path ends at height $0$, 
and (ii) we don't encounter more $D$'s than $U$'s in $m_1, \ldots, m_i$ for any $i$. 

In order to verify condition (i), we notice that if $m_i= U$, 
both edges out of the cells in row $i$ are directed toward the first column in $G_P$, 
whereas if $m_i =D$, both edges out of the cells in row $i$ are directed toward the 
second column in $G_P$.
If $m_i$ equals $H$ or $\widetilde{H}$, then there is an edge directed toward each column out of the cells of row $i$ in $G_P$.
Thus if there are $a$ $U$'s, $b$ $D$'s, and $c$ $H$'s and $d$ 
$\widetilde{H}$'s in $M$, then these account for $2a+c+d$ edges directed toward the first column 
in $G_P$ and 
$2b+c+d$ edges directed toward the second column in $G_P$.
The total in-degree of the first column is $n-1$ since every 
entry except $(1,1)$ has an edge of $G_P$ coming into it. 
Similarly, the total in-degree of the second column is $n$.
Since $M$ accounts for every edge in $G_{P}$ except $e$, which is directed towards 
the second column, it must be the case that $n-1=2a +c +d= 2b+c+d$. Hence   
$a=b$ and there are an equal number of $U$'s and 
$D$'s in $M$. Thus $M$ starts at (0,0) and ends at $(n-1,0)$. 

If condition (ii) fails, then let $i$ be the position of 
the first step to end beneath the $x$-axis, i.e. $m_i =D$ and 
there are an equal number of $U$'s and $D$'s 
among $m_1, \ldots, m_{i-1}$.  By our previous argument, there must be  
$i-1$ edges directed into second column from the cells in rows $1, \ldots, i-1$. 
Since edges that go from the first column to the second 
column go weakly downward and edges from the second column to 
the second column can go up by at most one, it follows that 
the edges that hit the second column from cells in rows $1, \ldots ,i-1$ 
must all end somewhere in the first $i$ rows. Now suppose 
that $(g,2)$, with $g \leq i$, is a cell that is not the endpoint of a directed 
edge from cells in rows $1, \ldots, i-1$.  It cannot be that 
$g=i$ since this would force $G_P$ to have edges $((i,1),(i,2))$ 
and $((i,2),(i+1,2))$. But this means that 
$P(i,2) = 1+P(i,1)$ and $P(i+1,2) = P(i,1)+2$ which is 
impossible since we must have have $P(i,1) < P(i+1,1) < P(i+1,2)$. 
Thus we must assume that $g < i$. Since 
$m_i$ is a down-edge and the edge that goes into cell $(i,2)$ can 
only have come from cell $(i-1,2)$, it must be the case that 
$G_P$ has edge $((i,1),(g,2))$. 
In this case, we must have the situation pictured in 
Figure \ref{figure:deg1}. That is, because $P$ is a column strict 
tableau, arrows that go from column 1 to column 2 cannot 
cross each other.  Hence all the cells of the form $(h,2)$ where 
$g < h \leq i$ must have been hit by vertical arrows.  But 
this is impossible because then $\{P(i,1),P(g,2),P(g+1,2), \ldots 
P(i,2),P(i+1,2)\}$ would be a set of consecutive elements so 
it would be impossible to have $P(i,1) < P(i+1,1) < P(i+1,2)$, as 
is required for $P$ to be a column strict tableau.

\fig{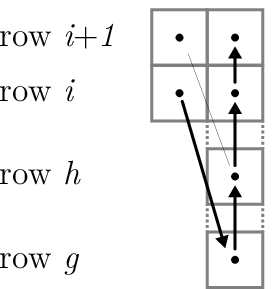}{The case where $((i,1),(g,2)) \in E(G_P)$ and $g < i$.}

Since we know that $C_n = |St_{2^n}| =  |\mathcal{M}^{(2)}_{n-1}|$, 
to prove that $\Gamma$ is bijection, we need only 
show that $\Gamma$ is 1:1.  Let $P_1$ and $P_2$ be 
two distinct column strict tableaux of shape $2^n$.  Let $s$ be 
the largest element such that $1, \ldots, s$ are in the same 
cells in both $P_1$ and $P_2$ and suppose that $s$ is in 
cell $(i,b)$.  Then without loss of generality, we can assume that $s+1$ is in column 1 
in $P_1$ and in column 2 in $P_2$. But this means that 
the arrow out of $(i,b)$ goes to column 1 in $G_{P_1}$ and to column 
2 in $G_{P_2}$, which implies that the $i$-th step 
in $\Gamma(P_1)$ is not equal to the $i$-th step in $\Gamma(P_2)$, 
so $\Gamma(P_1) \neq \Gamma(P_2)$.  Thus we have proved the following 
theorem.

\begin{theorem} For all $n \geq 2$, 
$\Gamma: St_{2^n} \rightarrow \mathcal{M}^{(2)}_{n-1}$
 is a bijection.
\end{theorem}

Note that $P \in St_{2^n}$ is degenerate if and only 
if there are no $\widetilde{H}$'s in $\Gamma(P)$.  Thus we have 
the following corollary. 
\begin{corollary}
For all $n \geq 2$, $|Deg_n| = M_{n-1}$ where $M_n$ is the 
number of Motzkin paths of length $n$. 
\end{corollary}

We say that two column strict tableaux $P,Q \in \mathcal{F}_{2,k}$ 
are {\em consecutive-Wilf equivalent} ($c$-Wilf equivalent) if 
$A^P(t) = A^Q(t)$ and we say that $P$ and $Q$ are 
{\em strongly $c$-Wilf equivalent} if $D^P(t,x) = D^Q(t,x)$. 
Clearly, it follows from Theorem \ref{thm:D} that 
if $P$ and $Q$ are degenerate column strict tableaux of 
shape $2^k$, then $P$ and $Q$ are strongly $c$-Wilf equivalent. 

We end this section with a few remarks about how we have 
derived formulas for $full_n^P$ for certain column strict tableaux $P$. 
First there are some simple equivalences which help 
us reduce the number of $P \in St_{2^k}$ that we 
have to consider.  Given any $F \in \mathcal{F}_{n,k}$, 
we define the generalized complement of $F$, $F^{gc}$ as a composition 
of three operations on $F$, $c$ for complement, $f_y$ for flipping 
the diagram about the $y$-axis, and $f_x$ for flipping the diagram 
about the $x$-axis. Here the operation of complement 
on $F$ is the result of replacing each element $i$ in $F$ by 
$kn+1 -i$.  The operation $f_y$ is the operation of 
taking a filling of the $k \times n$ rectangle with 
the numbers $1, \ldots, kn$, and flipping the diagram about 
the $y$-axis and the operation $f_x$ is the operation of 
taking a filling of the $k \times n$ rectangle with 
the numbers $1, \ldots, kn$, and flipping the diagram about 
the $x$-axis. We then define $F^{gc} = f_x\circ f_y \circ c (F)$. 
This process is pictured in Figure \ref{figure:flip}. 

\fig{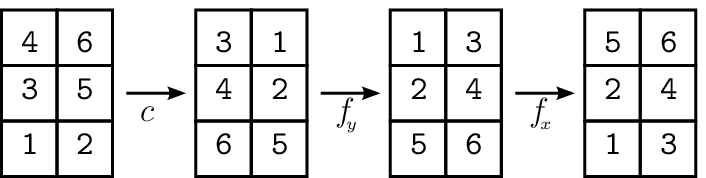}{The generalized complement of a diagram.}

It is easy to see that $F \in \mathcal{F}_{n,k}$ if and only if 
$F^{gc} \in \mathcal{F}_{n,k}$ and $P \in St_{2^k}$ if and only if 
$P^{gc} \in St_{2^k}$.  Moreover for 
any column strict tableau $P$ of shape $2^k$,  
$F \in \mathcal{F}_{n,k}$ has a $P$-match in columns $i$ and $i+1$  
 if and only if 
$F^{gc}$ has a $P^{gc}$-match in columns $n-i$ and $n+1-i$. 
It follows that $F \in Full_n^P$ if and only if 
$F^{gc} \in Full^{P^{gc}}_n$ so that 
$full_n^P = full_n^{P^{gc}}$. Thus $P$ and $P^{gc}$ are 
c-Wilf equivalent. Now in $\mathcal{F}_{2,3}$, 
the only non-degenerate column strict tableaux that are not complements of each other are the tableaux 
$P$ and $Q$ pictured in Figure \ref{figure:PQ}.

\fig{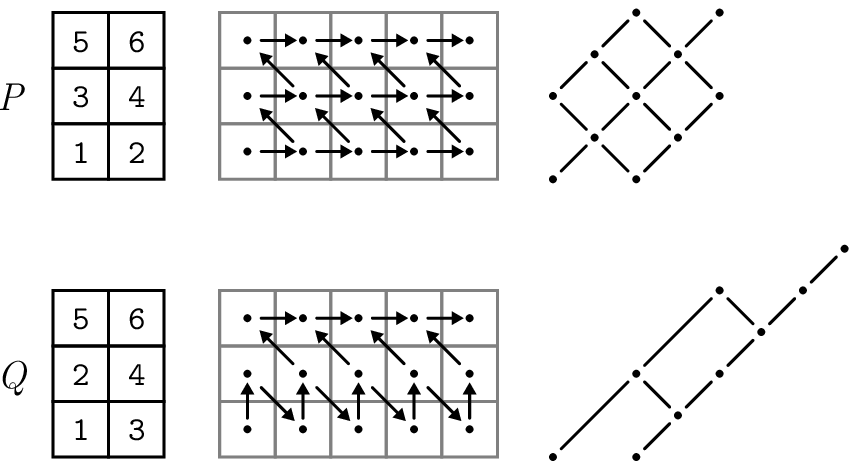}{The tableaux $P$ and $Q$ and their corresponding 
Hasse diagrams.}

Recall that in the proof of Theorem \ref{degenerate}, we observed 
that for any standard tableaux $P \in St_{2^k}$, 
the graph $G_{P,n}$ induces a poset 
$\mathcal{W}_{P,n} = (\{(i,j): 1 \leq i \leq k\ \& \ 1 \leq j \leq n\}, <_W)$ 
on the cells of the 
$k \times n$ rectangle by defining $(i,j) <_W (s,t)$ if 
and only if there is path from $(i,j)$ to $(s,t)$ in $G_{P,n}$ and 
that there is a one-to-one correspondence between the 
linear extensions of $\mathcal{W}_{P,n}$ and the elements 
of $Full_n^P$.

Now in the case of the standard tableaux $P$ and $Q$ in 
Figure \ref{figure:PQ}, we have pictured 
$G_{P,5}$ just to the right of $P$ and $G_{Q,5}$ just to the 
right of $Q$.  In the case of $P$, if we eliminate the vertices 
corresponding to (1,1) and (1,2), which must always be the first two elements  
in any linear extension of $\mathcal{W}_P(5)$, and we eliminate 
the vertices in cells (3,4) and (3,5), which must always be the last  
two elements  
in any linear extension of $\mathcal{W}_{P,5}$, we are left 
with the Hasse diagram pictured to its right. In the case of 
$Q$, if we eliminate the vertices corresponding to cells 
(1,1), (2,1), (1,2) and (2,2), which must always be the first four  
elements  
in any linear extension of $\mathcal{W}_{Q,5}$, and we eliminate 
the vertices in cells (3,4) and (3,5),  which must always be the last  
two elements  
in any linear extension of $\mathcal{W}_{Q,5}$, we are left 
with the Hasse diagram pictured to its right. 
We have developed general recursions of the number of linear extension 
of posets whose Hasse diagrams  
are similar to the one pictured for $Q$ in Figure \ref{figure:PQ} 
which have allowed us to show that 
$full_n^Q= \frac{1}{2n+1}\binom{3n}{n}$.  This work will appear in 
subsequent paper. However, we have not been able to find a 
formula for $full_n^P$.

\end{document}